\documentclass{amsart}

\usepackage[usenames,dvipsnames]{xcolor}
\usepackage{amsthm}
\usepackage{amsfonts}
\usepackage{amssymb}
\usepackage{tikz}
\usepackage{hyperref}
\usepackage[normalem]{ulem}

\newtheorem{theorem}{Theorem}[section]
\newtheorem{corollary}[theorem]{Corollary}

\newtheorem{proposition}[theorem]{Proposition}

\newtheorem{definition}[theorem]{Definition}
\newtheorem{remark}[theorem]{Remark}

\title{Splines in Geometry and Topology}
\author{Julianna Tymoczko}

\begin{document}
\begin{abstract}
This survey paper describes the role of splines in geometry and topology, emphasizing both similarities and differences from the classical treatment of splines.  The exposition is non-technical and contains many examples, with references to more thorough treatments of the subject.
\end{abstract}
\maketitle

The goal of this survey paper is to describe how splines arise in geometry and topology.  Geometric splines usually appear under the name {\em GKM theory} after Goresky-Kottwitz-MacPherson, who developed them to compute what is called the {\em cohomology ring} of a geometric object.  Geometers and analysts ask many of the same questions about splines: what is their dimension? can we identify a basis?  can we find explicit formulas for the elements of the basis?  However geometric constraints can change the tone of these questions: the splines may satisfy various symmetries or have a basis satisfying certain conditions.  And some questions are specific to geometric splines: geometers particularly care about the multiplication table with respect to a given basis. 

In Section \ref{section: GKM theory} we discuss GKM theory, followed in Section \ref{section: examples} by some important families of geometric examples, some of which are well-known to students of analytic splines and some of which may be useful in future. Section \ref{section: tools} sketches some techniques that are natural from the perspective of a geometer/topologist, including Morse flows and symmetries that come from geometric representation theory.  Finally in Section \ref{section: generalized splines} we generalize splines to a more abstract ring setting, both as a useful conceptual framework and because it provides new combinatorial tools.

This paper is targeted at researchers in geometric design, especially those with an analytic background. Our aim is to give an overview of theoretical tools and techniques from geometry and topology; we often illustrate concepts by example and refer to the literature for details on technical aspects.  A reader with a different mathematical perspective may be interested in surveys like \cite{Hol08, Tym05, Tym08a}.

\section{GKM theory}\label{section: GKM theory}

Cohomology is an algebraic gadget associated to a geometric object $X$ that encodes various properties of $X$.  Among other things, the cohomology of $X$ indicates the dimension of $X$, the number of connected components (how many separate pieces $X$ has), how many holes $X$ has, whether $X$ has singularities, and how different subspaces of $X$ intersect. 

We could treat very general kinds of geometric objects but for simplicity in this survey we take $X$ to be a compact complex manifold.  When we say that cohomology is an ``algebraic gadget" the most general interpretation is that cohomology is a ring.  In the cases of interest here, the cohomology ring is actually an algebra, namely a vector space in which one can multiply vectors.  The technical condition we assume is that {\em cohomology has coefficients in} $\mathbb{Q}, \mathbb{R}$, or especially $\mathbb{C}$.

In fact we will consider an enhanced version of cohomology called $T$-equivariant cohomology.  Equivariant cohomology has strictly more information than ordinary cohomology yet surprisingly can be easier to compute.  In our case $T$ is a torus, namely the group
\[T = \mathbb{C}^* \times \mathbb{C}^* \times \cdots \times \mathbb{C}^*\]
where $\mathbb{C}^*$ denotes the group of nonzero complex numbers.  

\begin{remark}
In the literature algebraic geometers tend to use the torus above while differential geometers tend to use a product of circles $T=S^1 \times \cdots \times S^1$.  This makes very little difference in practice because $\mathbb{C}^*$ contracts to $S^1$.  In other words, thinking of ${\mathbb{C}^*}$ as the plane with the origin removed, we can squeeze it 
thinner and thinner until only $S^1$ remains.  This squeezing process respects the constructions that underly cohomology.  The outcome after formalizing is that in cases where either $S^1 \times \cdots \times S^1$ or $\mathbb{C}^* \times \mathbb{C}^* \times \cdots \times \mathbb{C}^*$ can be used, the equivariant cohomology is the same.
\end{remark}

We require $T$ to act on $X$ ``nicely", which for us means:
\begin{enumerate}
\item $X$ has a finite number of $T$-fixed points
\item $X$ has a finite number of one-(complex)-dimensional $T$-orbits 
\item $X$ is ``equivariantly formal"
\end{enumerate}
The first two are geometric conditions that are relatively straightforward to check in specific examples.  The last is a technical condition that amounts to a certain spectral sequence degenerating.  In practice one usually assumes one of the many conditions that {\em implies} equivariant formality, for instance any of the following conditions:
\begin{itemize}
\item $X$ is a smooth complex projective algebraic variety
\item $X$ has no odd-dimensional ordinary cohomology
\item $X$ has a $T$-stable cell decomposition
\end{itemize}

We refer to this set of hypotheses as the {\em GKM hypotheses}.

\begin{definition}
If $X$ is a compact complex manifold with the action of a torus $T = \mathbb{C}^* \times \mathbb{C}^* \times \cdots \times \mathbb{C}^*$ that satisfies (1)--(3) then we say $X$ and $T$ satisfy the {\em GKM hypotheses}.
\end{definition}

\subsection{Examples of varieties with ``nice" torus actions}\label{subsection: projective space moment graph}

\subsubsection{The projective line $\mathbb{P}^1$}  Consider the collection of lines through the origin in the plane, which we call $\mathbb{P}^1$.  Each point $(x,y)$ in the plane other than the origin gives rise to a point in $\mathbb{P}^1$ by taking the unique line through $(x,y)$ and $(0,0)$.  Of course  many points give rise to the same line; in fact multiplying the vector $(x,y)$ by a nonzero constant won't change its direction and thus won't change the line through the origin on which it lies.  We use square brackets to indicate points in $\mathbb{P}^1$ itself, which are equivalence classes of points in the plane.  In other words we could write
\[\mathbb{P}^1 = \frac{\{ [x,y] : x,y \in \mathbb{C} \textup{ and at least one of }x,y \textup{ is nonzero}\}}{[x,y] \sim [\lambda x , \lambda y] \textup{ for all nonzero } \lambda \in \mathbb{C}^*}\]
to indicate that two different equations for a line are equivalent if they differ by a nonzero constant multiple.

There are different ways for a torus to act on $\mathbb{P}^1$.  Of interest to us is the following action of $T=\mathbb{C}^*$ on $\mathbb{P}^1$.  Suppose that $t \in \mathbb{C}^*$ and that $[x,y] \in \mathbb{P}^1$.  Then 
\[t \cdot [x,y] = [x,ty] \]
At this point the reader should pause and try to identify the fixed points of $\mathbb{P}^1$ under this torus action as well as the (only) one-dimensional orbit.

Notice that $[x,0]$ is fixed by this $T$-action.  Since $x \neq 0$ we conclude that $[x,0]$ is actually the single point $[1,0]$ by definition of $\mathbb{P}^1$.  In other words 
\[[x,0] \sim [1,0]\]
In fact $[x,0] \sim [x',0]$ for any nonzero $x,x' \in \mathbb{C}^*$ since $[x',0] = [\frac{x'}{x}x,\frac{x'}{x}0]$.  

A similar argument shows that $[0,y]$ is the single point $[0,1]$ in $\mathbb{P}^1$.  It is also a $T$-fixed point, since $t \cdot [0,1] = [0,t]$ by definition of the group action and $[0,t] \sim [0,1]$ by the  argument above.

In fact these are the only $T$-fixed points.  Indeed suppose that $[x,y]$ is $T$-fixed, meaning that $[x,ty] \sim [x,y]$ for every possible $t \in \mathbb{C}^*$.  Then there exists a nonzero scalar $\lambda$ with $[\lambda x, \lambda y] = [x,ty]$.  On the one hand $\lambda x = x$ so either $\lambda$ is one or $x$ is zero.  On the other hand $\lambda y = ty$ so either $\lambda = t$ or $y$ is zero.  Since $t$ is arbitrary we conclude that one of $x$ or $y$ must be zero.

In this example there is exactly one one-dimensional $T$-orbit consisting of all points $[x,y]$ with both $x \neq 0, y \neq 0$.  Indeed any such point $[x,y]$ can be written $[x,y] \sim [1,y']$ by multiplying by $\lambda = 1/x$.  Any two points $[1,y']$ and $[1,y'']$ are in the same $T$-orbit because we can choose $t = y'/y''$.  (The choices of $\lambda$ and $t$ are valid because all of $x$, $y'$, and $y''$ are nonzero.)

\subsubsection{The projective plane $\mathbb{P}^2$}   The projective plane $\mathbb{P}^2$ consists of all lines through the origin in $\mathbb{C}^3$.  We can describe it similarly to the projective line:
\[\mathbb{P}^2 = \frac{\{ [x_1,x_2, x_3] : x_1,x_2,x_3 \in \mathbb{C} \textup{ and at least one of }x_1,x_2,x_3 \textup{ is nonzero}\}}{[x_1,x_2, x_3] \sim [\lambda x_1,\lambda x_2, \lambda x_3] \textup{ for all nonzero } \lambda \in \mathbb{C}^*}\]
We choose a different torus action for $\mathbb{P}^2$ than $\mathbb{P}^1$ in order to lay the foundations for general projective space $\mathbb{P}^n$.  However the analysis is very similar.

Take $T = \mathbb{C}^* \times \mathbb{C}^* \times \mathbb{C}^*$ and define an action of $(t_1,t_2,t_3) \in T$ on $[x_1,x_2,x_3] \in \mathbb{P}^2$ by
\[(t_1,t_2,t_3)  \cdot [x_1,x_2,x_3] = [t_1x_1,t_2x_2,t_3x_3]\]
Again the reader should pause to identify $T$-fixed points and one-dimensional orbits before continuing with the exposition.

By the same argument as above we can see that the three points $[1,0,0]$, $[0,1,0]$, and $[0,0,1]$ are all fixed by this $T$-action and in fact that they are the only  $T$-fixed points.  (If $[x_1,x_2,x_3]$ is a point with at least two nonzero entries, we can restrict attention to the two nonzero entries and literally use the previous argument to conclude that the point cannot be fixed.)

Identifying one-dimensional orbits seems more complicated for $\mathbb{P}^2$ than for $\mathbb{P}^1$.  However suppose we consider a subset of $\mathbb{P}^2$ in which exactly two entries are nonzero, for instance 
\[\{[x_1,x_2,0]: x_1, x_2 \textup{ are nonzero in } \mathbb{C}\}\]
We can divide by the nonzero coordinate $x_1$ to get
\[\{[x_1,x_2,0]: x_1, x_2 \textup{ are nonzero in } \mathbb{C}\} \sim \{[1,x,0]: x \textup{ is nonzero in } \mathbb{C}\}\]
where $x = \frac{x_2}{x_1}$.  The torus acts on each element of this subset by
\[ (t_1,t_2,t_3) \cdot [1,x,0] = [t_1,t_2x,0]\]
and by previous arguments we know that $[t_1,t_2x,0] \sim [1,\frac{t_2}{t_1}x,0]$.  In other words $T$ preserves the set $\{[x_1,x_2,0]: x_1, x_2 \in \mathbb{C}^*\}$.  Of course {\em which} two entries were nonzero was immaterial.  We conclude that $\mathbb{P}^2$ has at least three one-dimensional $T$-orbits, each obtained by choosing  two entries to be nonzero.

In fact those three are the only one-dimensional $T$-orbits.  If $x_1,x_2,x_3$ are all nonzero then 
\[[x_1,x_2,x_3] \sim [1,\frac{x_2}{x_1}, \frac{x_3}{x_1}]\]
The image under an arbitrary element of the torus is
\[(t_1,t_2,t_3) \cdot [1,\frac{x_2}{x_1}, \frac{x_3}{x_1}] = [t_1, t_2 \frac{x_2}{x_1}, t_3 \frac{x_3}{x_1}]\]
which is the same point in $\mathbb{P}^2$ as $[1, \frac{t_2}{t_1}\frac{x_2}{x_1},\frac{t_3}{t_1} \frac{x_3}{x_1}]$.  The last two coordinates can be any pair of nonzero complex numbers because $t_1$, $t_2$, and $t_3$ are all arbitrary nonzero complex numbers.  In other words if all three entries are nonzero then $[x_1,x_2,x_3]$ lies on a two-dimensional $T$-orbit.

\subsubsection{Projective space $\mathbb{P}^{n-1}$}  More generally projective space $\mathbb{P}^{n-1}$ consists of all lines through the origin in $\mathbb{C}^{n}$.  As before we can write
\[\mathbb{P}^n = \frac{\{ [x_1,x_2, \ldots, x_n] : x_1,x_2,\ldots,x_n \in \mathbb{C} \textup{ and at least one of }x_1,x_2,\ldots,x_n \textup{ is nonzero}\}}{[x_1,x_2, \ldots,x_n] \sim [\lambda x_1,\lambda x_2, \ldots,\lambda x_n] \textup{ for all nonzero } \lambda \in \mathbb{C}^*}\]
Fix the torus $T= \mathbb{C}^* \times \mathbb{C}^ \times \cdots \times \mathbb{C}^*$ with $n$ copies of $\mathbb{C}^*$ and define an action of $T$ on $\mathbb{P}^{n-1}$ by 
\[(t_1,t_2,\ldots,t_n)  \cdot [x_1,x_2,\ldots,x_n] = [t_1x_1,t_2x_2,\ldots,t_nx_n]\]
The same analysis as before shows that the $T$-fixed points of $\mathbb{P}^{n-1}$ are the $n$ coordinate points 
\[[1,0,0,\ldots,0], [0,1,0,\ldots,0],\ldots,[0,0,0,\ldots,1]\] 
in which exactly one entry is nonzero. Each one-dimensional orbit consist of points in which exactly two entries are nonzero; the set of pairs $(i,j)$ with $1 \leq i < j \leq n$ indexes the set of one-dimensional orbits, determining which entries are nonzero.  

\subsection{Moment graph}

In $\mathbb{P}^2$ we saw that every one-dimensional $T$-orbit essentially looked like the one-dimensional $T$-orbit in $\mathbb{P}^1$.  In fact the closure of each one-dimensional $T$-orbit is exactly $\mathbb{P}^1$.  This holds not just for $\mathbb{P}^n$ but for all $X$ that satisfy the GKM hypotheses.

\begin{proposition} Suppose $X$ and $T$ satisfy the GKM hypotheses and $O$ is any one-dimensional $T$-orbit $O$ in $X$.  Then the closure of $O$ is homeomorphic to $\mathbb{P}^1$ and consists of $O$ together with two $T$-fixed points $N$ and $S$ (essentially the north and south poles).  
\end{proposition}

This means we can record the data of $T$-fixed points and one-dimensional orbits more concisely: as a graph. 

\begin{definition} 
Suppose $X$ and $T$ satisfy the GKM hypotheses.  The {\em moment graph} or {\em GKM graph} of $X$ is defined from the $T$-orbits of $X$ as follows:
\begin{center}
\begin{tabular}{lcl}
$T$-fixed points & $\leftrightarrow$ & vertices \\
$1$-dimensional $T$-orbits & $\leftrightarrow$ & edges \\
weight of $T$-action & $\leftrightarrow$ & label on edge
\end{tabular}
\end{center}
\end{definition}

(Section \ref{subsection: GKM theorem and discussion} explains how the {\em moment} gets into moment graph.)

The weight of the $T$-action essentially records the direction of the $T$-action.  For instance when studying the one-dimensional $T$-orbits in $\mathbb{P}^2$ we found that
\[ (t_1,t_2,t_3) \cdot [1,x,0] \sim [1, \frac{t_2}{t_1}x,0]\]
In other words the torus pushed the point $[1,x,0]$ in the $\frac{t_2}{t_1}$ direction.  We typically linearize this and (with a slight abuse of notation) record the weight as $t_2-t_1$.  Figure \ref{figure: moment graph examples} gives the moment graphs for $\mathbb{P}^1$ and $\mathbb{P}^2$.  In the moment graph for $\mathbb{P}^{n-1}$ the edge indexed by $(i,j)$ for $1 \leq i < j \leq n$ is labeled $t_j - t_i$.

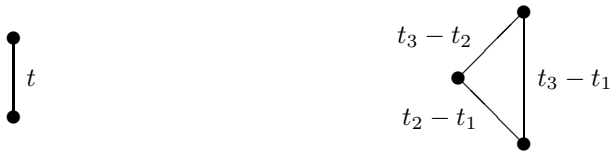
\begin{figure}[h]
\begin{picture}(10,50)(0,-10)
\put(0,0){\circle*{5}}
\put(0,30){\circle*{5}}
\put(0,0){\line(0,1){30}}
\put(5,12){$t$}
\end{picture}
\hspace{1.75in}
\begin{picture}(50,50)(-50,0)
\put(0,0){\circle*{5}}
\put(0,50){\circle*{5}}
\put(-25,25){\circle*{5}}

\put(0,0){\line(0,1){50}}
\put(0,0){\line(-1,1){25}}
\put(-25,25){\line(1,1){25}}

\put(5,22){$t_3-t_1$}
\put(-46,7){$t_2-t_1$}
\put(-48,38){$t_3-t_2$}
\end{picture}
\caption{Examples of moment graphs for $\mathbb{P}^1$ and $\mathbb{P}^2$} \label{figure: moment graph examples}
\end{figure}

\subsection{GKM theorem}\label{subsection: GKM theorem and discussion}

We can now give Goresky-Kottwitz-MacPherson's theorem \cite{GKM98}.

\begin{theorem}[Goresky-Kottwitz-MacPherson]
Let $X$ be a compact complex manifold with the action of a torus $T$ that satisfies the GKM hypotheses.  Let $G_X$ be its moment graph.  Then the equivariant cohomology ring $H^*_T(X; \mathbb{C})$ satisfies the ring isomorphism
\[H^*_T(X; \mathbb{C}) \cong \{p \in \mathbb{C}[t_1,\ldots,t_n]^{|V|}: \textup{ for each edge } uv \textup{ the difference }p_u - p_v \textup{ is in } \alpha(uv)\}\] 
\end{theorem}

From the perspective of splines, the point is the following:

\begin{corollary}
Suppose $X$ is a compact complex manifold with a torus action that satisfies the GKM hypotheses.  The equivariant cohomology of $X$ is isomorphic to the ring of complex splines $S^0_\infty(G_X)$.
\end{corollary}

What precisely does this mean?  Classically the splines $S^r_d(\Delta)$ consist of $\mathcal{C}^r$ piecewise polynomials of degree at most $d$ on the combinatorial object $\Delta$.  We take $\Delta$ to be a pure polyhedral complex whose maximal faces are all $d$-dimensional polytopes in $\mathbb{R}^d$ and use the example of a triangulation of a region of the plane. 

The graph $G_X$ is the dual graph to $\Delta$, namely $G_X$ has one vertex for each $d$-dimensional face and an edge between two vertices if the corresponding faces intersect in a $d-1$-dimensional face.  For instance if $\Delta$ is a triangulation then $G_X$ is the classical dual graph: it has one vertex per triangle and an edge between vertices if the corresponding triangles share an edge.  

We can label the edges of $G_X$ as well.  If $\sigma_1 \cap \sigma_2$ is a $d-1$-dimensional face in $\Delta$ then the affine span of $\sigma_1 \cap \sigma_2$ is a hyperplane.   Label the edge corresponding to $\sigma_1 \cap \sigma_2$ in $G_X$ with the affine form that vanishes on that hyperplane.  For instance when $\Delta$ is a triangulation the label on each edge of $G_X$ is essentially the slope of the line between the corresponding triangles.  Section \ref{subsection: Alfeld split} shows an example in detail.  

Billera first used dual graphs in the context of splines, proving that the ring of splines on the dual graph is isomorphic to $S^r(\Delta)$ when $\Delta$ is a {\em hereditary} pure polyhedral complex \cite{Bil88}.  (Hereditary complexes are essentially those without holes or pinchpoints; for instance any triangulation of a convex region in $\mathbb{R}^2$ is hereditary.  For details see e.g. \cite{BilRos91, BilRos92}.)

The GKM perspective uses graphs because of the historical evolution of the result.  Pieces of Goresky-Kottwitz-MacPherson's approach predate them significantly.  The $T$-orbits on $X$ had been studied extensively in different contexts: when $X$ is a symplectic manifold with a Hamiltonian $T$-action there's a natural map $\mu: X \rightarrow \mathbb{R}^n$ called {\em the moment map} that linearizes the action of the torus, ``straightening out" the torus orbits.  Kostant noticed that the image of the moment map for one important manifold was a convex polytope \cite{Kos70}, leading to work by Atiyah, Guillemin-Sternberg, and Heckman that collectively proved that the image of the moment map is a union of convex polytopes \cite{Ati82, GuiSte82, Hec82}.  In fact the points and $1$-dimensional faces of the moment map form precisely the moment graph that arises in GKM theory.  

Another crucial step of the GKM proof is to construct a map from the equivariant cohomology $H^*_T(X)$ into a product of polynomial rings.  It turns out that the equivariant cohomology of a point is simply $\mathbb{C}[t_1,\ldots,t_n]$ and in fact the natural inclusion $\iota: X^T \hookrightarrow X$ of $T$-fixed points into $X$ induces the desired map on cohomology $\iota^*: H^*_T(X) \rightarrow \mathbb{C}[t_1,\ldots,t_n]^{|V|}$.  This map plays an especially important role in symplectic geometry and topology, in the work of Kirwan \cite{Kir84}, among others.  Chang-Skjelbred analyzed the image of this map in very general settings, where the image is determined by a complicated spectral sequence \cite{ChaSkj74}.  One of Goresky-Kottwitz-MacPherson's deep contributions is to identify a package of GKM hypotheses for which the general picture simplifies so dramatically.

At the same time, some work on equivariant cohomology leads to the classical spline perspective, especially when $X$ is a toric variety. Each toric variety $X$ is associated to a polytope, making it more natural to consider piecewise polynomials on the faces of the polytope.  This is what Payne \cite{Pay06} and Bahri-Franz-Ray \cite{BFR09} do.  

I have not found indication in the literature that any of these geometers or topologists were aware that they had recreated the classical notion of splines.  Schenck appears to be the first person to work in both splines and equivariant cohomology \cite{Sch12}.

\begin{remark}[Singularities]
For simplicity we assumed that $X$ is a manifold.  In fact GKM theory applies to many singular varieties as well, including Schubert varieties (see Sections \ref{subsection: Grassmannians} and \ref{subsection: flag varieties} below).  Geometers often find this surprising, but the GKM hypotheses reduce the calculation of equivariant cohomology to an analysis of zero- and one-dimensional orbits, where geometric singularities don't significantly arise.

Singularities in the sense of analytic splines correspond to faces like the subgraph shown in Figure \ref{figure: singular spline}.  
\begin{figure}[h]
\begin{picture}(50,50)(-25,-25)
\put(-25,25){\circle*{5}}
\put(25,25){\circle*{5}}
\put(0,0){\circle*{5}}
\put(0,50){\circle*{5}}

\put(0,0){\line(1,1){25}}
\put(-25,25){\line(1,1){25}}
\put(0,0){\line(-1,1){25}}
\put(25,25){\line(-1,1){25}}

\put(-25,5){$\alpha$}
\put(14,40){$\alpha$}
\put(15,5){$\beta$}
\put(-23,38){$\beta$}
\end{picture}
\hspace{1in}
\begin{picture}(100,100)(-50,-25)
\put(-25,25){\circle*{5}}
\put(25,25){\circle*{5}}
\put(0,0){\circle*{5}}
\put(0,50){\circle*{5}}

\put(0,0){\line(1,1){25}}
\put(-25,25){\line(1,1){25}}
\put(0,0){\line(-1,1){25}}
\put(25,25){\line(-1,1){25}}

\put(-50,-25){\circle*{5}}
\put(50,75){\circle*{5}}
\multiput(-49,-24)(4,2){12}{\circle*{1}}
\multiput(-49,-24)(2,4){12}{\circle*{1}}
\multiput(-49,-24)(4,6){13}{\circle*{1}}
\multiput(-49,-24)(6,4){13}{\circle*{1}}

\multiput(49,74)(-4,-2){12}{\circle*{1}}
\multiput(48,73.5)(-4,-2){12}{\circle*{1}}
\multiput(47,73)(-4,-2){11}{\circle*{1}}

\multiput(49,74)(-2,-4){12}{\circle*{1}}
\multiput(48.5,73)(-2,-4){12}{\circle*{1}}
\multiput(48,72)(-2,-4){11}{\circle*{1}}

\multiput(49,74)(-4,-6){13}{\circle*{1}}
\multiput(48.3,73)(-4,-6){13}{\circle*{1}}
\multiput(47.6,72)(-4,-6){13}{\circle*{1}}

\multiput(49,74)(-6,-4){13}{\circle*{1}}
\multiput(48,73.3)(-6,-4){13}{\circle*{1}}
\multiput(47,72.6)(-6,-4){13}{\circle*{1}}

\put(-21,7){\small $\alpha$}
\put(13,38){\small $\alpha$}
\put(13,6){\small $\beta$}
\put(-21,36){\small $\beta$}
\end{picture}
\caption{The subgraph corresponding to a singular vertex and the moment graph for the Grassmannian $G(2,4)$}\label{figure: singular spline}
\end{figure}
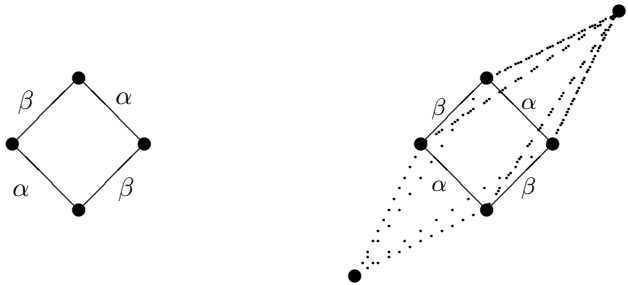
They have no relationship to geometric singularities of the variety $X$.  Singular faces appear frequently in moment graphs; for instance Figure \ref{figure: singular spline} also shows the moment graph of the Grassmannian $G(2,4)$.  Section \ref{subsection: Grassmannians} describes the Grassmannian in more detail; the edge-labels of the moment graph of $G(2,4)$ can be found in many other sources \cite{KnuTao03, Tym05}.
\end{remark}

The difference in perspective between students of analytic splines and what we will call geometric splines affects the kinds of data we assume and the kinds of questions we ask about splines.
\begin{itemize}
\item Geometers typically consider the splines $H^*_T(X,\mathbb{C})$ {\em as modules} over the polynomial ring $\mathbb{C}[t_1,\ldots,t_n]$ rather than vector spaces because the GKM construction does not naturally impose any degree constraints on splines.  In addition the GKM hypotheses guarantee that $H^*_T(X,\mathbb{C})$ is a free module over the polynomial ring $\mathbb{C}[t_1,\ldots,t_n]$.  (In some contexts this free-ness is taken as the definition of equivariant formality.)
\item Geometers typically know the dimension of the splines $H^*_T(X,\mathbb{C})$ as a module over the polynomial ring for geometric reasons.  
\item Geometers, like analysts, want explicit formulas for bases for the space of splines.  In many cases the torus action can be used to construct geometric bases or to suggest preferred forms for bases.  In other cases a group action on the spline module can be used to construct representation-theoretic bases.  But these constructions guarantee {\em existence} of a basis or set of splines rather than concrete descriptions that can be used in calculations.  Section \ref{subsection: flag varieties} says more about bases when $X$ is a flag variety and Section \ref{section: tools} discusses bases in general.
\item Geometers care about the multiplicative structure of the ring of splines.  In particular geometers want explicit multiplication tables in terms of a basis for the splines.  Section \ref{section: examples} describes some context behind this problem in the case of ordinary and equivariant cohomology of specific families of manifolds $X$.
\item Geometers focus on $S^0_{\infty}(G_X)$ since these are the splines isomorphic to $H^*_T(X; \mathbb{C})$.  Billera asked what kind of geometric meaning the rings $S^0_d(G_X)$ carry as $d$ increases. 
\end{itemize}
In the rest of this survey we give important geometric examples of splines and elaborate on these points.

\section{Important examples for geometers/topologists}\label{section: examples}

This section describes some of the most important examples for geometers and topologists.  Some are already used in analytic splines; projective space turns out to be essentially the Alfeld split.  We present others both to expand on the  questions that concluded Section \ref{section: GKM theory} and to suggest tools that may be useful within analytic splines.

\subsection{Projective spaces}\label{subsection: Alfeld split}  We described projective space extensively in Section \ref{subsection: projective space moment graph}.  The moment graph of projective space is better-known within analytic splines as the Alfeld split \cite{Alf84} or sometimes the Clough-Tocher split (but see the comment in Kolesnikov-Sorokina \cite[Section 2]{KolSor14}).  We are particularly grateful to Alexei Kolesnikov and Gwyneth Whieldon for pointing this out.

We describe the Alfeld split following Kolesnikov-Sorokina \cite{KolSor14} and Schenck \cite{Sch14}.  Let  $T = [v_1,\ldots,v_{n+1}]$ be an $n$-simplex in $\mathbb{R}^n$.  Given a central interior vertex $v_0$ the Alfeld split is defined by taking the cone over the boundary.   Without loss of generality we can assume that $v_0$ is the origin, each $v_i$ is the standard basis vector $e_i$, and $v_{n+1}=-\sum_{i=1}^ne_i$.  Then there are $n+1$ facets $F_1,F_2,\ldots,F_{n+1}$ of the Alfeld split, each containing all but one of the vertices $v_0, v_1, \ldots, v_{n+1}$.  Specifically for each $i=1,\ldots,n+1$ the $i^{th}$ facet is the face that does not contain $v_i$.  The intersection of two facets $F_i$ and $F_j$ lies on the hyperplane defined by 
\[F_i \cap F_j = \left\{ \begin{array}{ll}
x_i-x_j=0 &\textup{ if } 1 \leq i < j \leq n \\
x_i=0 &\textup{ if } 1 \leq i < n+1 \textup{ and } j = n+1
\end{array} \right. \]
Thus the dual graph of the Alfeld split is the complete graph on $n+1$ vertices indexed without loss of generality by $\{F_1,F_2,\ldots,F_{n+1}\}$ with edge $F_i \leftrightarrow F_j$ labeled as above.

\begin{proposition}
The ring of $C^0$ splines on the moment graph of projective space $\mathbb{P}^n$ is isomorphic to the ring of splines on the Alfeld split $AS^n$. 
\end{proposition}

\begin{proof}
For each $i=1,\ldots,n$ let $\alpha_i = t_{i+1}-t_i$.  Define a map 
\[\varphi: \mathbb{R}[\alpha_1, \alpha_2,\ldots,\alpha_n] \rightarrow  \mathbb{R}[x_1, x_2,\ldots,x_n]\] on the generators $\alpha_1, \ldots, \alpha_n$ by
\[\varphi(\alpha_i) = x_{i}-x_{i+1} \textup{   if   } i \in \{1,2,\ldots,n-1\} \]
and
\[\varphi(\alpha_n) = x_{n}\]
The map $\varphi$ extends to a ring homomorphism.  It is surjective because for each $i=1,\ldots,n-1$ we have
\[x_i = \varphi(\alpha_i + \alpha_{i+1} + \cdots + \alpha_n)\]
This also shows that the map $\varphi$ sends the moment graph of $\mathbb{P}^n$ to the dual graph of the Alfeld split, in the sense that it sends each edge-label to the corresponding edge-label.  

The map $\varphi$ is injective because $\{x_1-x_2,x_2-x_3,\ldots,x_{n-1}-x_n,x_n\}$ are algebraically independent.  Indeed a polynomial $p(\alpha_1,\ldots,\alpha_n)$ is in the kernel of $\varphi$ only if $\varphi(p) = 0$.  But $\varphi(p) = p(x_1-x_2,\ldots,x_n)$ is zero only if $p$ is identically zero.

In other words $\varphi$ is an isomorphism of polynomial rings that sends the moment graph of $\mathbb{P}^n$ to that of the dual graph to the Alfeld split.  It follows from Billera's result \cite{Bil88} that the ring of splines on the moment graph of $\mathbb{P}^n$ is isomorphic to the ring of splines on the Alfeld split, as desired.
\end{proof}

\subsection{Grassmannians}\label{subsection: Grassmannians} The Grassmannian is a natural generalization of projective space in which we consider higher-dimensional subspaces of $\mathbb{C}^n$ than lines.  More precisely the Grassmannian $G(k,n)$ consists of the set of $k$-dimensional vector subspaces of $\mathbb{C}^n$.  Note that $G(1,n)$ is another way to write the projective space $\mathbb{P}^{n-1}$.  

Grassmannians occur naturally in different geometric contexts.  For instance the tangent bundle of a smooth $k$-dimensional manifold $M \subseteq \mathbb{C}^n$ can be thought of as a map from $M$ to the Grassmannian $G(k,n)$.  Explicitly each point $p \in M$ is mapped to the $k$-dimensional subspace obtained by translating the tangent space to $M$ at $p$ to the origin.

The cohomology ring of the Grassmannian is particularly interesting. Just like we identified the one-dimensional orbits of $\mathbb{P}^2$ by choosing which entries were nonzero, we can partition the Grassmannian into a natural collection of subspaces by choosing which entries in the $k$-dimensional planes are pivots.  The sets of pivots are indexed by partitions of $n$ and the collection of subspaces corresponding to the partition $\lambda$ is called a {\em Schubert variety}.  The Schubert varieties induce a basis of {\em Schubert classes} $\sigma_\lambda$ for the cohomology ring $H^*(G(k,n), \mathbb{C})$.  Since they are a basis, we can write the products $\sigma_{\lambda} \cdot \sigma_{\mu}$ again in terms of the basis of Schubert classes:
\[\sigma_{\lambda} \cdot \sigma_{\mu} = \sum_{\nu} c_{\lambda, \mu}^{\nu} \sigma_{\nu}\]
The entries $c_{\lambda, \mu}^{\nu}$ in this multiplication table have remarkable properties.
\begin{itemize}
\item The entries $c_{\lambda, \mu}^{\nu}$ count the intersections of linear spaces in a vector space.  The $19^{th}$-century mathematician Hermann Schubert, a pioneer of enumerative geometry, computed these intersections by hand.  His name is now given to {\em Schubert calculus}, which is more generally considered the study of the cohomology ring of the Grassmannian and related spaces.
\item The entries $c_{\lambda, \mu}^{\nu}$ also describe the ring of symmetric polynomials.  A polynomial is a {\em symmetric polynomial} if it is invariant under all permutations of the variables.  One main direction of research in algebraic combinatorics studies the ring of symmetric polynomials with respect to different bases, including an important basis of {\em Schur functions} $s_{\lambda}$.  It turns out that writing a product of Schur functions in terms of Schur functions gives
$s_{\lambda} \cdot s_{\mu} = \sum_{\nu} c_{\lambda, \mu}^{\nu} s_{\nu}$ for the very same $c_{\lambda, \mu}^{\nu}$  that appear in the cohomology of the Grassmannian.
\item The entries $c_{\lambda, \mu}^{\nu}$ also describe the representations of the group $GL_n(\mathbb{C})$ of $n \times n$ invertible matrices.  The irreducible representations $V_{\lambda}$ of $GL_n(\mathbb{C})$ are indexed by partitions, just like Schubert classes and Schur functions.  The tensor product $V_{\lambda} \otimes V_{\mu}$ of the vector space $V_{\lambda}$ and $V_{\mu}$ can  be decomposed into irreducible representations $V_{\lambda} \otimes V_{\mu} = \bigoplus_{\nu} c_{\lambda, \mu}^{\nu} V_{\nu}$.  The tensor product multiplicities also turn out to be precisely the coefficients in the multiplication table of $H^*(G(k,n), \mathbb{C})$.
\end{itemize}

The structure constants $c_{\lambda, \mu}^{\nu}$ for $H^*(G(k,n), \mathbb{C})$ are well-understood.  The classical combinatorial formula for $c_{\lambda, \mu}^{\nu}$ in terms of the underlying partitions $\lambda$, $\mu$, and $\nu$ was discovered by Littlewood-Richardson.   More recently Knutson-Tao discovered a combinatorial formula for the equivariant structure constants $c_{\lambda, \mu}^{\nu}$ for $H^*_T(G(k,n), \mathbb{C})$ using splines (in the form of GKM theory) \cite{KnuTao03}.  Their proof has three steps: 1) use splines to determine the base case of a recurrence, 2) prove that any set of numbers that satisfy the base recurrence are the same; and then 3) miraculously identify a second formula that satisfies the same recurrence and thus is the same as $c_{\lambda, \mu}^{\nu}$.  A particularly interesting element of the proof is that while splines play a key role, it is in some sense an indirect role.  Geometrically splines have proven very difficult to calculate.

\subsection{Flag variety}\label{subsection: flag varieties} Flag varieties form another important family of spaces that satisfy the GKM hypotheses.  Consider the general linear group $GL_n(\mathbb{C})$, namely the collection of $n \times n$ invertible matrices.  The subgroup of upper-triangular invertible matrices is called the Borel subgroup and typically denoted $B$.  The flag variety is the quotient $GL_n(\mathbb{C})/B$.

On the one hand a flag is a coset $gB$.  On the other hand a flag $V_{\bullet}$ can be realized geometrically as a nested collection of subspaces 
\[V_1 \subsetneq V_2 \subsetneq V_3 \subsetneq \cdots \subsetneq V_{n-1} \subsetneq \mathbb{C}^n\]  To see that these two descriptions are equivalent, write the columns of $g$ as vectors $v_1, v_2, \ldots, v_n$ where $\langle v_1, \ldots, v_i \rangle = V_i$.  The different matrix representatives for a given flag $V_{\bullet}$ are exactly the matrices in the coset $gB$.  Indeed the matrices in $gB$ contain all possible nonzero vectors in $V_1$ as their first column; their second column contains all vectors in $V_2$ that are linearly independent from $V_1$, and so on.

The moment graph of the flag variety is defined as follows:
\begin{itemize}
\item The vertices are the permutations $S_n$.
\item There is an edge between each pair of permutations $w \leftrightarrow (ij)w$ that are related by multiplication by a transposition.  (Here $(ij)$ indicates the permutation that exchanges $i$ and $j$ and leaves all other numbers fixed.)
\item The label on the edge $w \leftrightarrow (ij)w$ is $t_j-t_i$.  (Typically we assume $j>i$.)
\end{itemize}
For instance Figure \ref{figure: moment graph GL_3/B} shows the moment graph of the flag variety associated to $GL_3(\mathbb{C})$.
\begin{figure}[h]
\begin{picture}(75,100)(-35,0)
\put(0,15){\circle*{5}}
\put(-30,30){\circle*{5}}
\put(30,30){\circle*{5}}
\put(-30,60){\circle*{5}}
\put(30,60){\circle*{5}}
\put(0,75){\circle*{5}}

\put(0,15){\line(-2,1){30}}
\put(30,30){\line(-2,1){60}}
\put(30,60){\line(-2,1){30}}

\put(0,15){\line(2,1){30}}
\put(-30,30){\line(2,1){60}}
\put(-30,60){\line(2,1){30}}

\put(0,15){\line(0,1){60}}
\put(-30,30){\line(0,1){30}}
\put(30,30){\line(0,1){30}}

\put(-200,75){\bf Edge labels}
\put(-200,55){Negative slope: $t_2-t_1$}
\put(-200,40){Positive slope: $t_3-t_2$}
\put(-200,25){Vertical: $t_3-t_1$}

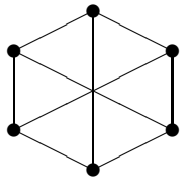
\put(75,75){\bf Permutations at vertices:}
\put(75,55){Top and bottom: $(13)$ and $e$}
\put(75,40){Leftmost: $(12)(23)$ and $(12)$}
\put(75,25){Rightmost: $(23)(12)$ and $(23)$}
\end{picture}
\caption{The moment graph of $GL_3(\mathbb{C})/B$} \label{figure: moment graph GL_3/B}
\end{figure}
The moment graph of the flag variety is also important in algebraic combinatorics, where it is called the Bruhat graph.

There are many symmetries in the moment graph of the flag variety and in the equivariant cohomology of the flag variety.  For instance, the permutation group $S_n$ acts on the vertices $S_n$ either by left- or right-multiplication.  At the same time, the group $S_n$ also acts by permuting the variables $t_1, t_2,\ldots, t_n$.  Some of these actions extend to $S_n$ actions on the splines, too, as we discuss in Section \ref{section: tools}.

The careful reader may also notice certain similarities between the moment graph of $\mathbb{P}^2$ and that of $GL_3(\mathbb{C})/B$.  In fact the moment graph of $GL_3(\mathbb{C})/B$ is a twisted product of the graph consisting of a single edge together with the moment graph of $\mathbb{P}^2$.  Figure \ref{figure: fiber bundle} shows how the single edge (drawn in boldface) is transported around the cycle with a M\"{o}bius-like twist on the back edges (drawn with dotted lines).  This is part of a general phenomenon studied extensively by Guillemin-Sabatini-Zara \cite{GSZ12, GSZ13}.  Each Grassmannian $G(k,n)$ can be identified with the quotient $GL_n(\mathbb{C})/P_k$ of the flag variety by the larger group $P_k \supset B$ of $n \times n$ invertible matrices whose bottom-left $n-k \times n-k$ block is zero. This gives rise to a fiber bundle $P_k/B \hookrightarrow GL_n(\mathbb{C})/B \rightarrow \hspace{-0.75em} \rightarrow G(k,n)$ whose equivariant cohomology satisfies the module isomorphism 
\[H^*_T(P/B) \otimes H^*_T(G(k,n)) \cong H^*_T(GL_n(\mathbb{C})/B)\]
This result is often called the equivariant Leray-Hirsch isomorphism; it can be realized as a straightforward product of splines in two different submodules of $H^*_T(GL_n(\mathbb{C})/B)$ \cite{DreTym}.  
\begin{figure}[h]


\begin{picture}(50,80)(-20,0)
\put(0,0){\circle*{5}}
\put(0,50){\circle*{5}}
\put(-25,25){\circle*{5}}

\put(0,0){\line(-1,1){25}}
\put(-25,25){\line(1,1){25}}

\put(50,25){\circle*{5}}
\put(50,75){\circle*{5}}
\put(25,50){\circle*{5}}

\multiput(0,0)(2,3){25}{\circle*{1}}
\multiput(50,25)(-4,2){12}{\circle*{1}}

\put(50,25){\line(-1,1){25}}
\put(25,50){\line(1,1){25}}

\thicklines
\thicklines
\put(50,25){\line(-2,-1){50}}
\put(50,75){\line(-2,-1){50}}
\put(25,50){\line(-2,-1){50}}

\end{picture}
\caption{The moment graph of $GL_3(\mathbb{C})/B$ drawn as a fiber bundle over $\mathbb{P}^2$} \label{figure: fiber bundle}
\end{figure}
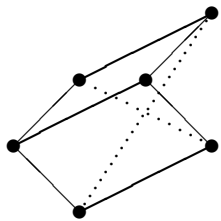

Just like projective space and Grassmannians, the flag variety has a natural collection of subvarieties called {\em Schubert varieties} that induce a basis for (ordinary and equivariant) cohomology.  There is also a remarkable formula for the polynomial associated to the restriction of the equivariant Schubert class $\sigma^v$ at the fixed point $w$.  This formula is often referred to as Billey's formula for $\sigma^v_w$ and was discovered by Andersen-Jantzen-Soergel and Billey independently \cite{AJS94, Bil99}.  We give the formula in the case of $GL_n(\mathbb{C})/B$ though it holds in more generality, including projective spaces and Grassmannians, so long as appropriate permutations are chosen to represent the Schubert classes in $H^*_T(G(k,n),\mathbb{C})$.  

\begin{theorem}[Billey's formula]
Write $w$ as a product of simple reflections $(i,i+1)$ in as few reflections as possible.  Denote this factorization of $w$ by $(i_1,i_1+1)(i_2,i_2+1)\cdots (i_k,i_k+1)$.  Let $R$ be the collection of subsets $\{j_1<j_2<\cdots < j_{\ell}\}$ of minimum cardinality for which $v = (i_{j_1},i_{j_1}+1)(i_{j_2},i_{j_2}+1)\cdots (i_{j_k},i_{j_k}+1)$.  Then 
\[ \sigma^v_w = \sum_{\{j_1<j_2<\cdots < j_{\ell}\} \in R} \prod_{r=1}^{\ell}  (i_1,i_1+1)(i_2,i_2+1)(i_3,i_3+1)\cdots (i_{j_r-1},i_{j_r-1}+1)(t_{j_r}-t_{j_r+1})\]
where permutations act on variables by exchanging the indices.
\end{theorem}
For instance $w=(13)$ could be written either $(12)(23)(12)$ or $(23)(12)(23)$.  If we pick $w=(12)(23)(12)$ and $v=(23)$ then the set $R$ contains only $\{2\}$ and 
\[\sigma^v_w = (12)(t_3-t_2)=t_3-t_1\]
If instead we take $w=(23)(12)(23)$ and $v=(23)$ then the set $R$ contains $\{1\}$ and $\{3\}$ and 
\[\sigma^v_w = (t_3-t_2) + (23)(12)(t_3-t_2)=t_3-t_2 + t_2-t_1=t_3-t_1\]
This example demonstrates that it is nontrivial even to see that the formula is well-defined. Figure \ref{figure: Billey's formula} gives all of the equivariant Schubert classes $\sigma^v$ in $H^*_T(GL_3(\mathbb{C})/B, \mathbb{C})$ listed in order 
\[v=e, (12), (23), (13), (12)(23), (23)(12)\]  
Each spline $\sigma^v$ in Figure \ref{figure: Billey's formula} is denoted by a copy of the graph with vertex $w$ labeled by the polynomial $\sigma^v_w$.  The top vertex of the second spline is labeled with the polynomial $t_3-t_1$ that we just computed.
\begin{figure}[h]
\begin{picture}(75,100)(-35,0)
\put(0,15){\circle*{5}}
\put(-30,30){\circle*{5}}
\put(30,30){\circle*{5}}
\put(-30,60){\circle*{5}}
\put(30,60){\circle*{5}}
\put(0,75){\circle*{5}}

\put(0,15){\line(-2,1){30}}
\put(30,30){\line(-2,1){60}}
\put(30,60){\line(-2,1){30}}

\put(0,15){\line(2,1){30}}
\put(-30,30){\line(2,1){60}}
\put(-30,60){\line(2,1){30}}

\put(0,15){\line(0,1){60}}
\put(-30,30){\line(0,1){30}}
\put(30,30){\line(0,1){30}}

\put(-2,0){\small $1$}
\put(-30,17){\small $1$}
\put(20,17){\small $1$}
\put(-30,67){\small $1$}
\put(20,67){\small $1$}
\put(-2,78){\small $1$}
\end{picture}
\hspace{0.35in}
\begin{picture}(75,100)(-35,0)
\put(0,15){\circle*{5}}
\put(-30,30){\circle*{5}}
\put(30,30){\circle*{5}}
\put(-30,60){\circle*{5}}
\put(30,60){\circle*{5}}
\put(0,75){\circle*{5}}

\put(0,15){\line(-2,1){30}}
\put(30,30){\line(-2,1){60}}
\put(30,60){\line(-2,1){30}}

\put(0,15){\line(2,1){30}}
\put(-30,30){\line(2,1){60}}
\put(-30,60){\line(2,1){30}}

\put(0,15){\line(0,1){60}}
\put(-30,30){\line(0,1){30}}
\put(30,30){\line(0,1){30}}

\put(-2,0){\small $0$}
\put(-45,15){\small $t_2-t_1$}
\put(20,17){\small $0$}
\put(-45,70){\small $t_2-t_1$}
\put(20,67){\small $t_3-t_1$}
\put(-10,80){\small $t_3-t_1$}
\end{picture}
\hspace{0.35in}
\begin{picture}(75,100)(-35,0)
\put(0,15){\circle*{5}}
\put(-30,30){\circle*{5}}
\put(30,30){\circle*{5}}
\put(-30,60){\circle*{5}}
\put(30,60){\circle*{5}}
\put(0,75){\circle*{5}}

\put(0,15){\line(-2,1){30}}
\put(30,30){\line(-2,1){60}}
\put(30,60){\line(-2,1){30}}

\put(0,15){\line(2,1){30}}
\put(-30,30){\line(2,1){60}}
\put(-30,60){\line(2,1){30}}

\put(0,15){\line(0,1){60}}
\put(-30,30){\line(0,1){30}}
\put(30,30){\line(0,1){30}}

\put(-2,0){\small $0$}
\put(-30,17){\small $0$}
\put(20,17){\small $t_3-t_2$}
\put(-45,70){\small $t_3-t_1$}
\put(20,67){\small $t_3-t_2$}
\put(-10,80){\small $t_3-t_1$}
\end{picture}
\hspace{0.35in}
\begin{picture}(75,100)(-35,0)
\put(0,15){\circle*{5}}
\put(-30,30){\circle*{5}}
\put(30,30){\circle*{5}}
\put(-30,60){\circle*{5}}
\put(30,60){\circle*{5}}
\put(0,75){\circle*{5}}

\put(0,15){\line(-2,1){30}}
\put(30,30){\line(-2,1){60}}
\put(30,60){\line(-2,1){30}}

\put(0,15){\line(2,1){30}}
\put(-30,30){\line(2,1){60}}
\put(-30,60){\line(2,1){30}}

\put(0,15){\line(0,1){60}}
\put(-30,30){\line(0,1){30}}
\put(30,30){\line(0,1){30}}

\put(-2,0){\small $0$}
\put(-30,17){\small $0$}
\put(20,17){\small $0$}
\put(-30,67){\small $0$}
\put(20,67){\small $0$}
\put(-42,80){\tiny $(t_2-t_1)(t_3-t_2)(t_3-t_1)$}
\end{picture}

\begin{picture}(75,100)(-35,0)
\put(0,15){\circle*{5}}
\put(-30,30){\circle*{5}}
\put(30,30){\circle*{5}}
\put(-30,60){\circle*{5}}
\put(30,60){\circle*{5}}
\put(0,75){\circle*{5}}

\put(0,15){\line(-2,1){30}}
\put(30,30){\line(-2,1){60}}
\put(30,60){\line(-2,1){30}}

\put(0,15){\line(2,1){30}}
\put(-30,30){\line(2,1){60}}
\put(-30,60){\line(2,1){30}}

\put(0,15){\line(0,1){60}}
\put(-30,30){\line(0,1){30}}
\put(30,30){\line(0,1){30}}

\put(-2,0){\small $0$}
\put(-30,17){\small $0$}
\put(20,17){\small $0$}
\put(-100,68){\small $(t_2-t_1)(t_3-t_1)$}
\put(20,67){\small $0$}
\put(-35,82){\small $(t_2-t_1)(t_3-t_1)$}
\end{picture}
\hspace{0.75in}
\begin{picture}(75,100)(-35,0)
\put(0,15){\circle*{5}}
\put(-30,30){\circle*{5}}
\put(30,30){\circle*{5}}
\put(-30,60){\circle*{5}}
\put(30,60){\circle*{5}}
\put(0,75){\circle*{5}}

\put(0,15){\line(-2,1){30}}
\put(30,30){\line(-2,1){60}}
\put(30,60){\line(-2,1){30}}

\put(0,15){\line(2,1){30}}
\put(-30,30){\line(2,1){60}}
\put(-30,60){\line(2,1){30}}

\put(0,15){\line(0,1){60}}
\put(-30,30){\line(0,1){30}}
\put(30,30){\line(0,1){30}}

\put(-2,0){\small $0$}
\put(-30,17){\small $0$}
\put(20,17){\small $0$}
\put(-30,67){\small $0$}
\put(21,68){\small $(t_3-t_2)(t_3-t_1)$}
\put(-35,82){\small $(t_3-t_2)(t_3-t_1)$}
\end{picture}
\caption{A module basis for $H^*_T(GL_3(\mathbb{C})/B, \mathbb{C})$ obtained using Billey's formula} \label{figure: Billey's formula}
\end{figure}
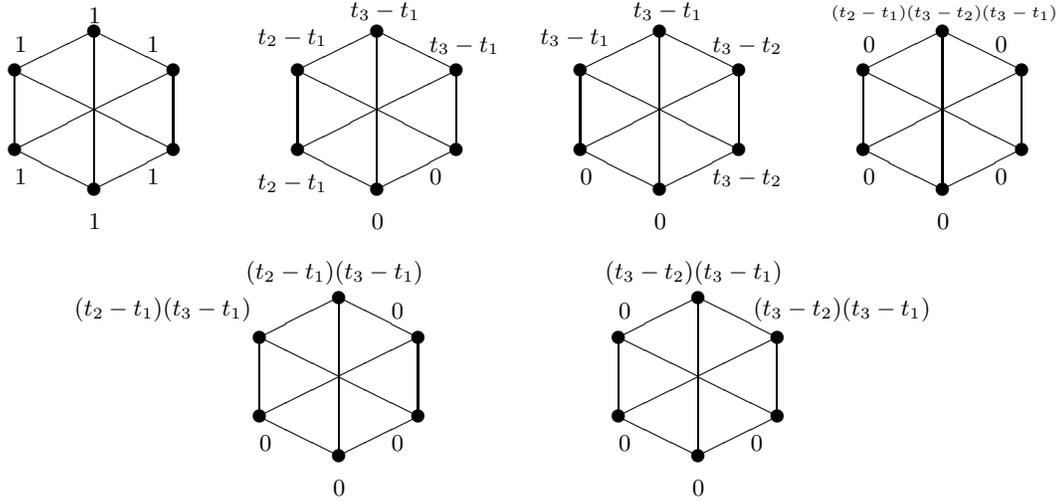

We close this section with an open question.  Recall that there is a closed formula for coefficients in the multiplication table for the cohomology $H^*(G(k,n),\mathbb{C})$ of each Grassmannian.  Perhaps surprisingly there is no good formula for the analogue in the case of the flag variety, either in ordinary or equivariant cohomology.  What does ``good" mean?  It means an explicit, positive formula---essentially a formula that counts some quantity rather than restating Billey's formula with inclusion/exclusion.  This is a problem that has attracted enormous attention within Schubert calculus with only limited success.  We suspect that the problem needs radically new tools or perspective.

\section{Geometric and topological tools for computing with splines}\label{section: tools}

In this section we describe three very natural tools for geometers and topologists working on splines.  We focus on how these tools are used to identify bases for the module of splines.  

It bears repeating that one significant difference from the analytic approach to splines is that geometers typically think about splines as modules {\em over a polynomial ring} rather than over the base field $\mathbb{C}$ or $\mathbb{R}$.  Geometers also tend to use bases of splines whose polynomials $p_u$ are either zero or homogeneous of a fixed degree depending on $p$.  This means that each module basis spline corresponds to a family of vector-space basis splines.  For instance in $S^1_2(\Delta)$ the constant module basis element ${\bf 1}$ gives rise to the six vector-space basis elements $\{ {\bf 1}, x{\bf 1}, y{\bf 1}, x^2{\bf 1}, xy{\bf 1}, y^2{\bf 1}\}$ while a degree-two basis element $p$ gives rise only to itself as a vector-space basis element.  (If there were degree-one basis elements $p$ they would generally correspond to the three vector-space basis elements $\{p, xp, yp\}$.) In some settings one needs to worry about whether the module is actually free over the polynomial ring, which would imply that minimal generating sets of splines are actually bases.  Spaces that satisfy the GKM hypotheses are always free modules so this is not a concern for GKM theory.

We now discuss some geometric strategies to identify module generators for splines. 

Recall in what follows that geometers traditionally study $S^0_{\infty}(G_X)$ since those are the splines isomorphic to equivariant cohomology (though Section \ref{section: generalized splines} describes interpretations of the ring $S^r_d(G_X)$).

\subsection{Toric actions and flow-up bases}\label{subsection: torus flow-up bases}

Given a torus acting on a space in $X$, one of the most natural things for geometers and topologists to do is to push subspaces around in $X$ using the torus.  This technique goes by different names in different fields (e.g. Morse theory for analysts/topologists, Bialynicki-Birula decompositions for algebraists/geometers). Regardless of the name, the techniques have consequences for splines that can be described precisely without reference to the underlying geometry.  

These techniques often use a one-dimensional torus chosen ``suitably generically" from a possibly larger torus action on $X$.  From a combinatorial perspective, choosing a one-dimensional torus corresponds to directing the edges of the graph $G_X$.  The ``suitably generic" condition helps ensure that the directed graph is a poset, namely $G_X$ has no directed cycles. 

These techniques give an analogue for splines of upper-triangularity, though the definition we give here is independent of geometry.

\begin{definition}\label{definition: flow-up}
Suppose that $G$ is a directed, edge-labeled graph.  A spline $p$ on $G$ is poset upper-triangular if there is a vertex $v$ in the graph for which $p_v \neq 0$ and $p_u = 0$ for all vertices $u$ in the graph without a directed path to $v$.  (The second condition can be written $u \not > v$ with the partial order $>$ defined by the direction on the graph.)  A basis $\mathcal{B}$ for the set of splines on $G$ is a flow-up basis if its elements can be indexed with distinct vertices so that the basis elements are all poset upper-triangular.  
\end{definition}

Intuitively each element of a flow-up basis collects the points that ``flow up" from a specific vertex according to the action of the torus.  

For some graphs, like cycles, poset upper-triangularity is equivalent to ordinary upper-triangularity.  Figure \ref{figure: flow-up basis P} shows a flow-up basis for the equivariant cohomology of $\mathbb{P}^2$.  Each spline is denoted by a copy of the graph with vertices labeled by polynomials.  For the reader's convenience, we left the edge-labels on the graph; they are indicated in grey.
\begin{figure}[h]
\begin{picture}(10,80)(0,-15)
\put(0,0){\circle*{5}}
\put(0,50){\circle*{5}}
\put(25,25){\circle*{5}}

\put(0,0){\line(0,1){50}}
\put(0,0){\line(1,1){25}}
\put(25,25){\line(-1,1){25}}

\put(-37,22){\color{gray} $t_3-t_1$}
\put(17,7){\color{gray} $t_2-t_1$}
\put(18,38){\color{gray} $t_3-t_2$}

\put(-3,-14){$1$}
\put(-3,56){$1$}
\put(30,21){$1$}
\end{picture}
\hspace{2in}
\begin{picture}(10,80)(0,-15)
\put(0,0){\circle*{5}}
\put(0,50){\circle*{5}}
\put(25,25){\circle*{5}}

\put(0,0){\line(0,1){50}}
\put(0,0){\line(1,1){25}}
\put(25,25){\line(-1,1){25}}

\put(-37,22){\color{gray} $t_3-t_1$}
\put(17,7){\color{gray} $t_2-t_1$}
\put(18,38){\color{gray} $t_3-t_2$}

\put(-3,-14){$0$}
\put(-13,56){$t_2-t_1$}
\put(30,21){$t_3-t_1$}
\end{picture}
\hspace{2in}
\begin{picture}(10,80)(0,-15)
\put(0,0){\circle*{5}}
\put(0,50){\circle*{5}}
\put(25,25){\circle*{5}}

\put(0,0){\line(0,1){50}}
\put(0,0){\line(1,1){25}}
\put(25,25){\line(-1,1){25}}

\put(-37,22){\color{gray} $t_3-t_1$}
\put(17,7){\color{gray} $t_2-t_1$}
\put(18,38){\color{gray} $t_3-t_2$}

\put(-3,-14){$0$}
\put(-28,56){$(t_3-t_2)(t_3-t_1)$}
\put(30,21){$0$}
\end{picture}
\caption{Flow-up basis for $H^*_T(\mathbb{P}^2,\mathbb{C})$}\label{figure: flow-up basis P}
\end{figure}
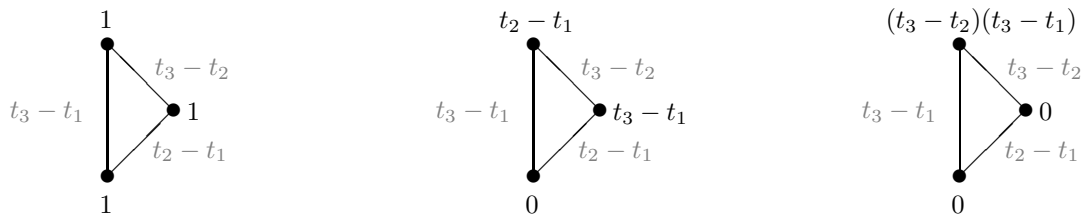
Flow-up bases for more general graphs do not correspond as naturally to upper-triangular bases.  For instance Figure \ref{figure: Billey's formula} shows a flow-up basis for $H^*_T(GL_3(\mathbb{C})/B, \mathbb{C})$.  We can impose a total order on the vertices to make this basis upper-triangular, but that requires an arbitrary choice; poset upper-triangularity is the more natural condition.  (For flag varieties, Grassmannians, and more general partial flag varieties, the flow-up basis is precisely the Schubert basis.)

Geometers generally require a stronger condition on flow-up classes than in Definition \ref{definition: flow-up}: that the flow-up class $p^v$ restricted to the fixed point $v$ is the product of the labels on the edges directed out of $v$.  Our definition applies in greater generality, discussed at greater length in Section \ref{section: generalized splines}.

Many natural spaces $X$ satisfy an additional geometric condition called {\em Palais-Smale}, which ensures that if a flow-up basis exists then it is unique.  Combinatorially the Palais-Smale condition says that the graph $G$ satisfies the property that for each directed edge $u \mapsto v$ the number of edges directed out of $u$ is greater than the number of edges directed out of $v$.   Informally the number of out-edges increases along each path up in the graph.  While the geometric Palais-Smale condition is stronger, the combinatorial Palais-Smale condition suffices to guarantee uniqueness of flow-up bases.  At the same time many natural manifolds are not Palais-Smale. For instance Figure \ref{figure: Hessenberg varieties} shows a family of manifolds called {\em regular semisimple Hessenberg varieties} parametrized by certain nondecreasing vectors recorded below each moment graph.  The third example from the left is not Palais-Smale.  Explicitly the vertex associated to the permutation $(12)$ has both the flow-up class in Figure \ref{figure: Billey's formula} and the flow-up class whose only nonzero labels are at vertices $(12)$ and $(12)(23)$, labeled by the polynomials $t_2-t_1$ and $t_2-t_3$ respectively.
\begin{figure}[h]
\begin{picture}(75,100)(-35,0)
\put(0,15){\circle*{5}}
\put(-30,30){\circle*{5}}
\put(30,30){\circle*{5}}
\put(-30,60){\circle*{5}}
\put(30,60){\circle*{5}}
\put(0,75){\circle*{5}}



\put(-30,-3){$h=(1,2,3)$}
\end{picture}
\hspace{0.35in}
\begin{picture}(75,100)(-35,0)
\put(0,15){\circle*{5}}
\put(-30,30){\circle*{5}}
\put(30,30){\circle*{5}}
\put(-30,60){\circle*{5}}
\put(30,60){\circle*{5}}
\put(0,75){\circle*{5}}

\put(0,15){\line(-2,1){30}}

\put(-30,60){\line(2,1){30}}

\put(30,30){\line(0,1){30}}

\put(-30,-3){$h=(2,2,3)$}
\end{picture}
\hspace{0.35in}
\begin{picture}(75,100)(-35,0)
\put(0,15){\circle*{5}}
\put(-30,30){\circle*{5}}
\put(30,30){\circle*{5}}
\put(-30,60){\circle*{5}}
\put(30,60){\circle*{5}}
\put(0,75){\circle*{5}}

\put(0,15){\line(-2,1){30}}
\put(30,60){\line(-2,1){30}}

\put(0,15){\line(2,1){30}}
\put(-30,60){\line(2,1){30}}

\put(-30,30){\line(0,1){30}}
\put(30,30){\line(0,1){30}}
\put(-30,-3){$h=(2,3,3)$}
\end{picture}
\hspace{0.35in}
\begin{picture}(75,100)(-35,0)
\put(0,15){\circle*{5}}
\put(-30,30){\circle*{5}}
\put(30,30){\circle*{5}}
\put(-30,60){\circle*{5}}
\put(30,60){\circle*{5}}
\put(0,75){\circle*{5}}

\put(0,15){\line(-2,1){30}}
\put(30,30){\line(-2,1){60}}
\put(30,60){\line(-2,1){30}}

\put(0,15){\line(2,1){30}}
\put(-30,30){\line(2,1){60}}
\put(-30,60){\line(2,1){30}}

\put(0,15){\line(0,1){60}}
\put(-30,30){\line(0,1){30}}
\put(30,30){\line(0,1){30}}
\put(-30,-2){$h=(3,3,3)$}
\end{picture}
\caption{Moment graphs for regular semisimple Hessenberg varieties (edges and vertices labeled as in Figure \ref{figure: moment graph GL_3/B})}\label{figure: Hessenberg varieties}
\end{figure}
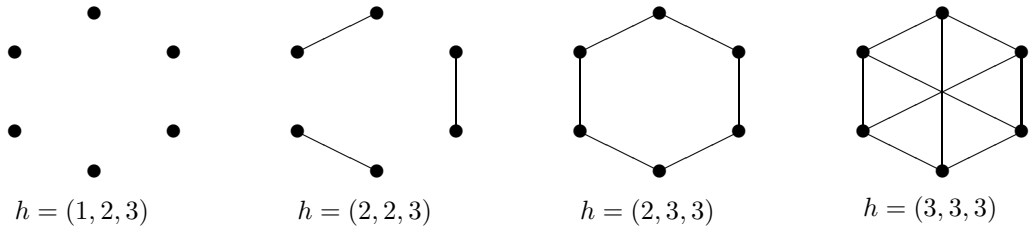

\subsection{Geometric representations and symmetrized bases}

A second geometric tool for working with splines comes out of geometric representation theory.  The fundamental idea is that graph automorphisms (namely bijections $V(G) \rightarrow V(G)$ that preserve the adjacencies in $G$) sometimes preserve the spline relations as well.  In this section we sketch the general philosophy using flag varieties; for a deeper discussion and more examples, see e.g. \cite{Tym08a}.  

Recall that the vertices in the moment graph of the flag variety are associated with the elements of the permutation group $S_n$.  Each element $v$ in the group $S_n$ acts on $S_n$ in two different ways: either by left multiplication $v \cdot w \mapsto vw$ or by right multiplication $v * w \mapsto wv^{-1}$.  Both of these actions give rise to an action on splines as follows.

Right-multiplication on the vertices induces an action on the splines in the equivariant cohomology ring because the labels on the edges come from left-multiplication by transpositions.  In other words the moment graph for the flag variety contains an edge $(ij)w \leftrightarrow w$ if and only if it contains an edge $(ij)wv^{-1} \leftrightarrow wv^{-1}$ for each $v \in S_n$.  Figure \ref{figure: right multiplication} shows what right-multiplication by the simple reflection $(23)$ does to a flow-up class.  
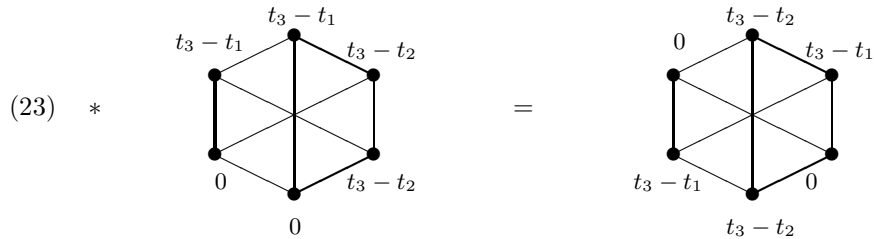
\begin{figure}[h]
\begin{picture}(30,100)(0,0)
\put(0,45){$(23)$}
\put(30,45){$*$} 
\end{picture}
\hspace{0.5in}
\begin{picture}(75,100)(-35,0)
\put(0,15){\circle*{5}}
\put(-30,30){\circle*{5}}
\put(30,30){\circle*{5}}
\put(-30,60){\circle*{5}}
\put(30,60){\circle*{5}}
\put(0,75){\circle*{5}}

\put(0,15){\line(-2,1){30}}
\put(30,30){\line(-2,1){60}}

\put(-30,30){\line(2,1){60}}
\put(-30,60){\line(2,1){30}}

\put(0,15){\line(0,1){60}}
\put(30,30){\line(0,1){30}}

\put(-2,0){\small $0$}
\put(-30,17){\small $0$}
\put(20,17){\small $t_3-t_2$}
\put(-45,70){\small $t_3-t_1$}
\put(20,67){\small $t_3-t_2$}
\put(-10,80){\small $t_3-t_1$}

\thicklines
\put(0,15){\line(2,1){30}}
\put(30,60){\line(-2,1){30}}
\put(-30,30){\line(0,1){30}}
\end{picture}
\hspace{0.5in}
\begin{picture}(20,100)(0,0)
\put(0,45){$=$}
\end{picture}
\hspace{0.4in}
\begin{picture}(75,100)(-35,0)
\put(0,15){\circle*{5}}
\put(-30,30){\circle*{5}}
\put(30,30){\circle*{5}}
\put(-30,60){\circle*{5}}
\put(30,60){\circle*{5}}
\put(0,75){\circle*{5}}

\put(0,15){\line(-2,1){30}}
\put(30,30){\line(-2,1){60}}

\put(-30,30){\line(2,1){60}}
\put(-30,60){\line(2,1){30}}

\put(0,15){\line(0,1){60}}
\put(30,30){\line(0,1){30}}

\put(-10,0){\small $t_3-t_2$}
\put(-45,17){\small $t_3-t_1$}
\put(20,17){\small $0$}
\put(-30,70){\small $0$}
\put(20,67){\small $t_3-t_1$}
\put(-10,80){\small $t_3-t_2$}

\thicklines
\put(0,15){\line(2,1){30}}
\put(30,60){\line(-2,1){30}}
\put(-30,30){\line(0,1){30}}
\end{picture}
\caption{Right-action of $S_3$ on splines in $H^*_T(GL_3(\mathbb{C})/B)$}\label{figure: right multiplication}
\end{figure}

Left-multiplication of the vertex labels by $(23)$ corresponds to exchanging labels on either side of the edges of positive slope.  That does not preserve splines: for instance it would send the flow-up class of Figure \ref{figure: right multiplication} to a class whose central vertical edge had $t_3-t_2$ on the bottom and $t_3-t_1$ on the top.  However if we also permute the variables by $(23)$ the outcome is again a spline, as shown in Figure \ref{figure: left multiplication}.  
\begin{figure}[h]
\begin{picture}(30,100)(0,0)
\put(0,45){$(23)$}
\put(30,48){\circle*{3}} 
\end{picture}
\hspace{0.5in}
\begin{picture}(75,100)(-35,0)
\put(0,15){\circle*{5}}
\put(-30,30){\circle*{5}}
\put(30,30){\circle*{5}}
\put(-30,60){\circle*{5}}
\put(30,60){\circle*{5}}
\put(0,75){\circle*{5}}

\put(0,15){\line(-2,1){30}}
\put(30,30){\line(-2,1){60}}
\put(30,60){\line(-2,1){30}}

\put(0,15){\line(0,1){60}}
\put(-30,30){\line(0,1){30}}
\put(30,30){\line(0,1){30}}

\put(-2,0){\small $0$}
\put(-30,17){\small $0$}
\put(20,17){\small $t_3-t_2$}
\put(-48,70){\small $t_3-t_1$}
\put(20,67){\small $t_3-t_2$}
\put(-10,80){\small $t_3-t_1$}

\thicklines
\put(0,15){\line(2,1){30}}
\put(-30,30){\line(2,1){60}}
\put(-30,60){\line(2,1){30}}
\end{picture}
\hspace{0.5in}
\begin{picture}(20,100)(0,0)
\put(0,45){$=$}
\end{picture}
\hspace{0.4in}
\begin{picture}(75,100)(-35,0)
\put(0,15){\circle*{5}}
\put(-30,30){\circle*{5}}
\put(30,30){\circle*{5}}
\put(-30,60){\circle*{5}}
\put(30,60){\circle*{5}}
\put(0,75){\circle*{5}}

\put(0,15){\line(-2,1){30}}
\put(30,30){\line(-2,1){60}}
\put(30,60){\line(-2,1){30}}

\put(0,15){\line(0,1){60}}
\put(-30,30){\line(0,1){30}}
\put(30,30){\line(0,1){30}}

\put(-10,0){\small $t_2-t_3$}
\put(-45,17){\small $t_2-t_3$}
\put(20,17){\small $0$}
\put(-48,70){\small $t_2-t_1$}
\put(20,67){\small $0$}
\put(-10,80){\small $t_2-t_1$}

\thicklines
\put(0,15){\line(2,1){30}}
\put(-30,30){\line(2,1){60}}
\put(-30,60){\line(2,1){30}}
\end{picture}
\caption{Left-action of $S_3$ on splines in $H^*_T(GL_3(\mathbb{C})/B)$}\label{figure: left multiplication}
\end{figure}
(This action on polynomials appeared in Billey's formula in Section \ref{subsection: flag varieties} as well.)

These geometric representations lend themselves to two different kinds of tools.  The first is a collection of degree-lowering operators on the collection of splines that behaves particularly well with respect to flow-up bases.  Bernstein-Gelfand-Gelfand and Demazure originally defined {\em divided difference operators} for flag varieties \cite{BGG73, Dem73}.  The essential formula for a divided difference operator $\partial_i$ is:
\[ \partial_i p = \frac{(i,i+1) \cdot p - p}{t_{i+1}-t_i}\]
As written the formula could apply to polynomials $p$ with $(i,i+1)$ exchanging the $i^{th}$ and $i+1^{th}$ variables, or it could apply to splines $p$ with the left action of $S_n$ defined above.  A similar formula can also be written for the right action but instead of dividing by the constant $t_{i+1}-t_i$ we divide by the label on the edge $v(i,i+1) \leftrightarrow v$ incident to each vertex $v$.  One key property divided difference operators satisfy for the flow-up (equivalently Schubert) basis $\{ \sigma_v: v \in S_n\}$ of the flag variety is
\[ \partial_i \sigma_v = \left\{ \begin{array}{l}
\sigma_{(i,i+1)v} \textup{    if $(i,i+1)v$ is below $v$ in the moment graph }\\
0 \textup{   otherwise}
\end{array} \right.\]
Divided difference operators are an important tool in algebraic combinatorics; for instance, part of the proof of Billey's formula is an elementary application of divided difference operators \cite{Bil99}.

The second tool has been much less studied.  Representations can also be used to create {\em symmetrized bases} for splines or subsets of splines.  In this case the group action can be used to transport a particular polynomial around the vertices, creating splines that are fixed by the overall $S_n$ action.  Figure \ref{figure: symmetrized basis for flag variety} shows a symmetrized basis for splines on the flag variety $GL_3(\mathbb{C})/B$. 
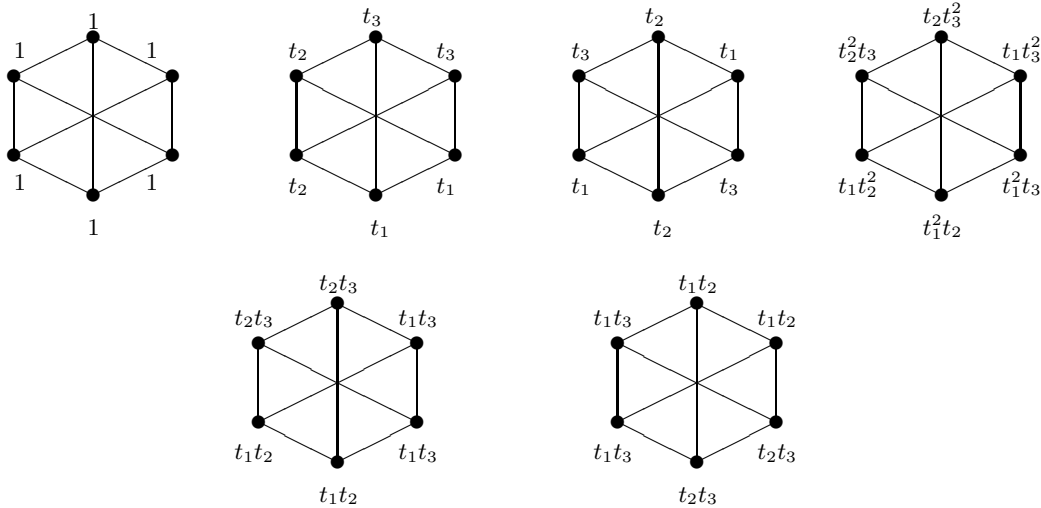
\begin{figure}[h]
\begin{picture}(75,100)(-35,0)
\put(0,15){\circle*{5}}
\put(-30,30){\circle*{5}}
\put(30,30){\circle*{5}}
\put(-30,60){\circle*{5}}
\put(30,60){\circle*{5}}
\put(0,75){\circle*{5}}

\put(0,15){\line(-2,1){30}}
\put(30,30){\line(-2,1){60}}
\put(30,60){\line(-2,1){30}}

\put(0,15){\line(2,1){30}}
\put(-30,30){\line(2,1){60}}
\put(-30,60){\line(2,1){30}}

\put(0,15){\line(0,1){60}}
\put(-30,30){\line(0,1){30}}
\put(30,30){\line(0,1){30}}

\put(-2,0){\small $1$}
\put(-30,17){\small $1$}
\put(20,17){\small $1$}
\put(-30,67){\small $1$}
\put(20,67){\small $1$}
\put(-2,78){\small $1$}
\end{picture}
\hspace{0.35in}
\begin{picture}(75,100)(-35,0)
\put(0,15){\circle*{5}}
\put(-30,30){\circle*{5}}
\put(30,30){\circle*{5}}
\put(-30,60){\circle*{5}}
\put(30,60){\circle*{5}}
\put(0,75){\circle*{5}}

\put(0,15){\line(-2,1){30}}
\put(30,30){\line(-2,1){60}}
\put(30,60){\line(-2,1){30}}

\put(0,15){\line(2,1){30}}
\put(-30,30){\line(2,1){60}}
\put(-30,60){\line(2,1){30}}

\put(0,15){\line(0,1){60}}
\put(-30,30){\line(0,1){30}}
\put(30,30){\line(0,1){30}}

\put(-2,0){\small $t_1$}
\put(-33,16){\small $t_2$}
\put(23,16){\small $t_1$}
\put(-33,67){\small $t_2$}
\put(23,67){\small $t_3$}
\put(-5,80){\small $t_3$}
\end{picture}
\hspace{0.35in}
\begin{picture}(75,100)(-35,0)
\put(0,15){\circle*{5}}
\put(-30,30){\circle*{5}}
\put(30,30){\circle*{5}}
\put(-30,60){\circle*{5}}
\put(30,60){\circle*{5}}
\put(0,75){\circle*{5}}

\put(0,15){\line(-2,1){30}}
\put(30,30){\line(-2,1){60}}
\put(30,60){\line(-2,1){30}}

\put(0,15){\line(2,1){30}}
\put(-30,30){\line(2,1){60}}
\put(-30,60){\line(2,1){30}}

\put(0,15){\line(0,1){60}}
\put(-30,30){\line(0,1){30}}
\put(30,30){\line(0,1){30}}

\put(-2,0){\small $t_2$}
\put(-33,16){\small $t_1$}
\put(23,16){\small $t_3$}
\put(-33,67){\small $t_3$}
\put(23,67){\small $t_1$}
\put(-5,80){\small $t_2$}
\end{picture}
\hspace{0.35in}
\begin{picture}(75,100)(-35,0)
\put(0,15){\circle*{5}}
\put(-30,30){\circle*{5}}
\put(30,30){\circle*{5}}
\put(-30,60){\circle*{5}}
\put(30,60){\circle*{5}}
\put(0,75){\circle*{5}}

\put(0,15){\line(-2,1){30}}
\put(30,30){\line(-2,1){60}}
\put(30,60){\line(-2,1){30}}

\put(0,15){\line(2,1){30}}
\put(-30,30){\line(2,1){60}}
\put(-30,60){\line(2,1){30}}

\put(0,15){\line(0,1){60}}
\put(-30,30){\line(0,1){30}}
\put(30,30){\line(0,1){30}}

\put(-7,0){\small $t_1^2t_2$}
\put(-39,16){\small $t_1t_2^2$}
\put(23,16){\small $t_1^2t_3$}
\put(-39,67){\small $t_2^2t_3$}
\put(23,67){\small $t_1t_3^2$}
\put(-7,80){\small $t_2t_3^2$}
\end{picture}

\begin{picture}(75,100)(-35,0)
\put(0,15){\circle*{5}}
\put(-30,30){\circle*{5}}
\put(30,30){\circle*{5}}
\put(-30,60){\circle*{5}}
\put(30,60){\circle*{5}}
\put(0,75){\circle*{5}}

\put(0,15){\line(-2,1){30}}
\put(30,30){\line(-2,1){60}}
\put(30,60){\line(-2,1){30}}

\put(0,15){\line(2,1){30}}
\put(-30,30){\line(2,1){60}}
\put(-30,60){\line(2,1){30}}

\put(0,15){\line(0,1){60}}
\put(-30,30){\line(0,1){30}}
\put(30,30){\line(0,1){30}}

\put(-7,0){\small $t_1t_2$}
\put(-39,16){\small $t_1t_2$}
\put(23,16){\small $t_1t_3$}
\put(-39,67){\small $t_2t_3$}
\put(23,67){\small $t_1t_3$}
\put(-7,80){\small $t_2t_3$}
\end{picture}
\hspace{0.75in}
\begin{picture}(75,100)(-35,0)
\put(0,15){\circle*{5}}
\put(-30,30){\circle*{5}}
\put(30,30){\circle*{5}}
\put(-30,60){\circle*{5}}
\put(30,60){\circle*{5}}
\put(0,75){\circle*{5}}

\put(0,15){\line(-2,1){30}}
\put(30,30){\line(-2,1){60}}
\put(30,60){\line(-2,1){30}}

\put(0,15){\line(2,1){30}}
\put(-30,30){\line(2,1){60}}
\put(-30,60){\line(2,1){30}}

\put(0,15){\line(0,1){60}}
\put(-30,30){\line(0,1){30}}
\put(30,30){\line(0,1){30}}

\put(-7,0){\small $t_2t_3$}
\put(-39,16){\small $t_1t_3$}
\put(23,16){\small $t_2t_3$}
\put(-39,67){\small $t_1t_3$}
\put(23,67){\small $t_1t_2$}
\put(-7,80){\small $t_1t_2$}
\end{picture}

\caption{Symmetrized basis for $H^*_T(GL_3(\mathbb{C})/B$}\label{figure: symmetrized basis for flag variety}
\end{figure}
In each case the polynomial at the bottom vertex was transported by the group action around all of the vertices. This $S_n$ action actually descends to Grassmannians as well.  Figure \ref{figure: symmetrized basis P} shows the symmetrized basis for splines on $\mathbb{P}^2$ as a comparison. Guillemin-Sabatini-Zara constructed examples of symmetrized bases as a step towards identifying bases for splines that arise as fiber bundles \cite{GSZ12, GSZ13}. 
\begin{figure}[h]
\begin{picture}(50,70)(-50,-10)
\put(0,0){\circle*{5}}
\put(0,50){\circle*{5}}
\put(-25,25){\circle*{5}}

\put(0,0){\line(0,1){50}}
\put(0,0){\line(-1,1){25}}
\put(-25,25){\line(1,1){25}}

\put(5,22){\color{gray} $t_3-t_1$}
\put(-46,7){\color{gray} $t_2-t_1$}
\put(-48,38){\color{gray} $t_3-t_2$}

\put(-3,-14){$1$}
\put(-3,56){$1$}
\put(-37,21){$1$}
\end{picture}
\hspace{1.5in}
\begin{picture}(50,70)(-50,-10)
\put(0,0){\circle*{5}}
\put(0,50){\circle*{5}}
\put(-25,25){\circle*{5}}

\put(0,0){\line(0,1){50}}
\put(0,0){\line(-1,1){25}}
\put(-25,25){\line(1,1){25}}

\put(5,22){\color{gray} $t_3-t_1$}
\put(-46,7){\color{gray} $t_2-t_1$}
\put(-48,38){\color{gray} $t_3-t_2$}

\put(-3,-14){$t_1$}
\put(-3,56){$t_3$}
\put(-38,21){$t_2$}
\end{picture}
\hspace{1.5in}
\begin{picture}(50,70)(-50,-10)
\put(0,0){\circle*{5}}
\put(0,50){\circle*{5}}
\put(-25,25){\circle*{5}}

\put(0,0){\line(0,1){50}}
\put(0,0){\line(-1,1){25}}
\put(-25,25){\line(1,1){25}}

\put(5,22){\color{gray} $t_3-t_1$}
\put(-46,7){\color{gray} $t_2-t_1$}
\put(-48,38){\color{gray} $t_3-t_2$}

\put(-4,-14){$t_1^2$}
\put(-4,57){$t_3^2$}
\put(-39,21){$t_2^2$}
\end{picture}
\caption{Symmetrized basis for $H^*_T(\mathbb{P}^2,\mathbb{C})$}\label{figure: symmetrized basis P}
\end{figure}
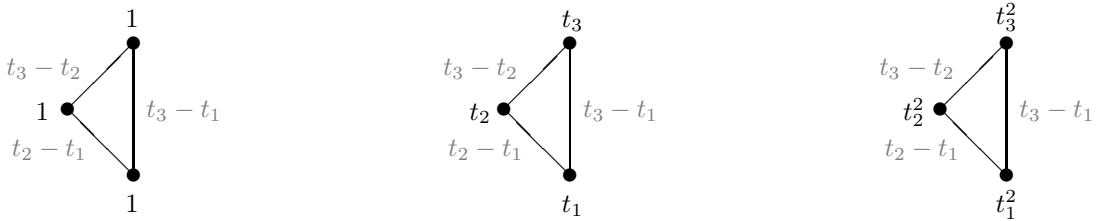 
 
Symmetrized bases are not nearly as common as flow-up bases in the literature. In one sense the construction is more algebraic than geometric since it relies on a representation on cohomology.   Moreover this symmetrization only describes the whole module if the representation on splines is trivial, which is not usually the case.   For instance the left action of $S_n$ restricts to the splines on the Hessenberg varieties shown in Figure \ref{figure: Hessenberg varieties}.  
\begin{figure}[h]
\begin{center}
\begin{tabular}{|c|c|}
\cline{1-2} Hessenberg function & Number of symmetrized basis classes \\
\cline{1-2} h=(3,3,3) & 6 \\
\cline{1-2} h=(2,3,3) & 4 \\
\cline{1-2} h=(2,2,3) & 2 \\
\cline{1-2} h=(1,2,3) & 1 \\
\hline
\end{tabular}
\end{center}
\caption{Symmetrized basis classes for splines corresponding to regular semisimple Hessenberg varieties}\label{figure: symmetrized basis Hess vars}
\end{figure}
In general the number of symmetrized basis elements decreases as the parameter $h$ changes, as shown in Table \ref{figure: symmetrized basis Hess vars}.  Currently the only known formula computes the number of symmetrized class rather than the classes themselves.  It is an open question to find an explicit formula for a maximal set of basis classes for the equivariant cohomology of regular semisimple Hessenberg varieties that is also symmetric under the $S_n$-action.  (These symmetrized basis classes generate a subring that is isomorphic to the cohomology of a different family of Hessenberg varieties called regular nilpotent Hessenberg varieties \cite{AHHM}.)

\subsection{Subvarieties}

Containment is another natural geometric relationship to consider, though one with more limited uses.  This has been most extensively considered for a GKM space $Y$ which is a subvariety of another GKM space $X$, both with respect to the same torus action.  Schubert varieties are one example, as are other examples that come from flowing subspaces as far as possible in one particular direction.  In these cases information about bases and geometric representations for $X$ can be applied directly to splines on $Y$.  

In general some of the most interesting results about subvarieties come from leveraging information about bases for splines of the ambient space.  For instance both the Grassmannian $G/P$ and the quotient $P/B$ can be viewed as submodules of the splines $H^*_T(GL_n(\mathbb{C})/B)$ by restricting the Schubert classes to specific fixed points in the moment graph.  Together with Billey's formula, this permits calculations that aren't as easy when $G/P$ or $P/B$ are viewed just in their own right \cite{DreTym}. 

The technique of restricting splines to the fixed points lying inside certain subvarieties has been used even when the subvariety $Y$ is not a GKM space with respect to the same torus action \cite{HarTym}.  This technique---called {\em poset pinball}---can be used to construct the $\mathbb{C}^*$-equivariant cohomology of $Y$.  In some sense this is a geometry/topology problem more than a problem about splines: it restricts splines to various subsets of vertices, but often the cohomology $H^*_{\mathbb{C}^*}(Y)$ is not a spline ring itself.  

\section{Algebraic generalization of splines}\label{section: generalized splines}

We conclude with an algebraic generalization of splines that gives a more abstract framework for some of the problems that we have considered here.  

To define this more general version of splines, fix a graph $G=(V,E)$.  For simplicity we assume that $G$ has a finite vertex set, no multiedges (so each pair of vertices has at most one edge between them), and no loops (so there are no edges from a vertex to itself).  Let $R$ be any commutative ring with identity.  A map $\alpha: E \rightarrow \{\textup{ ideals in }R\}$ is called an {\em edge-labeling} of the graph $G$.  

The ring of splines on the edge-labeled graph $(G,\alpha)$ is defined to be
\[R_{G,\alpha} = \{p \in R^{|V|}: \textup{ for each edge } uv \textup{ the difference }p_u - p_v \textup{ is in } \alpha(uv)\}\]

Different choices of parameters in $R_{G,\alpha}$ naturally recover the parameters $r$ and $d$ in the classical ring of splines $S^r_d(G)$.  The parameter $r$ arises from a map $\alpha$ that sends each edge $uv$ to an ideal $\langle \ell_{uv}^{r+1} \rangle$ where $\ell_{uv}$ is a homogenous linear form determined by the slope of the edge $uv$ in the standard way (see also Billera's dual graph construction \cite{Bil88}). The parameter $d$ arises by choosing the ring
\[R = \mathbb{R}[x_1,x_2,\ldots,x_n]/\langle \textup{ monomials of degree at least } d+1 \rangle\]
Intuitively (and somewhat inaccurately) this means: compute splines over the ring $\mathbb{R}[x_1,x_2,\ldots,x_n]$ and ``throw away" everything  whose degree is too big.  One subtlety is that the division algorithm does not hold for multivariable polynomial rings; the field of Gr\"{o}bner bases developed in order to deal with the complications that arise.  

Geometers have not studied the rings corresponding to $S^r_d(G)$ when $d$ is finite or $r$ is nonzero, though it seems likely that $S^0_d(G)$ carries useful geometric information (as Billera noted).  The geometric implications of varying $r$ are more mysterious, and interesting.

A few words on bases in this context are useful.  Again we consider module bases rather than vector space bases.  Indeed splines over an arbitrary  ring $R$ do not necessarily have vector-space bases (for instance when $R$ is the integers, described more below).  Of more concern, the module of splines may not even be free, both in interesting \cite{DiP12} and uninteresting (e.g. when $R$ has zero divisors) ways.

When we first considered flow-up bases in Section \ref{subsection: torus flow-up bases}, we said that geometers typically require the flow-up class $p^v$ restricted to the vertex $v$ to be the product of labels on edges directed out of $v$.  There are examples of flow-up bases for generalized splines for which this condition does not hold (see \cite{HMR}, who give an example over the integers).  If $X$ is a GKM space then Definition \ref{definition: flow-up} (and the GKM conditions) implies the product condition on $p^v_v$.

In fact one could generalize splines even more, along the lines of Braden-MacPherson's construction of intersection homology \cite{BraMac01}.  Their starting point is a graph with a module $M_v$ at each vertex.  At each edge there is a module $M_{uv}$ and maps $\varphi_u: M_u \rightarrow M_{uv}$ and $\varphi_v: M_u \rightarrow M_{uv}$.  A spline in this context is a collection $p \in \prod_v M_v$ so that for each edge $uv$
\[\varphi_u(p_u) = \varphi_v(p_v)\]
When each $M_v$ is $R$, each $M_{uv}$ is $R/\alpha(uv)$, and each map $\varphi_u: M_u \rightarrow M_{uv}$ is the standard quotient map, we recover the definition above.  

\subsection{Splines over the integers and over $\mathbb{Z}/m\mathbb{Z}$}

Splines over the integers have interesting connections to number theory.  In fact the Chinese Remainder theorem can be restated as a question about splines on the path with three vertices whose edge labels are the ideals $m\mathbb{Z}$ and $m'\mathbb{Z}$ shown in Figure \ref{figure: CRT}.  We can ask whether given integers $a$ and $a'$ associated to the vertices in Figure \ref{figure: CRT} can be extended to a spline on the entire graph. 
\begin{figure}[h]
\begin{picture}(100,30)(0,0)
\multiput(10,10)(45,0){3}{\circle*{5}}
\put(10,10){\line(1,0){90}}

\put(25,0){\color{gray} $m$}
\put(70,0){\color{gray} $m'$}

\put(7,15){$a$}
\put(97,15){$a'$}
\put(52,15){$x$}
\end{picture}
\caption{Splines and the Chinese Remainder theorem}\label{figure: CRT}
\end{figure}
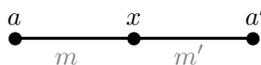
In other words is there an integer $x$ associated to the middle vertex so that
\[\left\{ \begin{array}{l}
x \equiv a \hspace{.25em} \textup{  mod } m \textup{ and } \\
x \equiv a' \textup{ mod }m' 
\end{array} \right.\]
This is precisely what the Chinese Remainder theorem determines.

As in the Chinese Remainder theorem, splines over $\mathbb{Z}$ are closely related to splines over $\mathbb{Z}/m\mathbb{Z}$.  We summarize what is currently known about splines over $\mathbb{Z}$ and $\mathbb{Z}/m\mathbb{Z}$.
\begin{itemize}
\item Splines over $\mathbb{Z}$ always form a free module of rank $|V(G)|$ \cite{GPT}.  (In other words they have a basis with the same number of elements as vertices in $G$.)  By contrast splines over $\mathbb{Z}/m\mathbb{Z}$ can have any rank between $1$ and $|V(G)|$ depending on the graph \cite{BowTym}.  Splines over $\mathbb{Z}/m\mathbb{Z}$ are not generally free because they are finite $\mathbb{Z}$-modules; like other finite $\mathbb{Z}$-modules, they have a natural analogue of a basis.
\item Given an edge-labeled graph $(G,\alpha)$, Bowden-Tymoczko identified the ring of splines over $\mathbb{Z}/m\mathbb{Z}$ using the structure theorem for finite abelian groups \cite{BowTym}.  Bowden-Philbin-Swift-Tammaro chose a particular $m$ and showed explicitly that the minimal representatives of the basis for splines over $\mathbb{Z}/m\mathbb{Z}$ also forms a basis for splines on $(G,\alpha)$ over $\mathbb{Z}$ \cite{BPST}.
\item Explicit bases for splines are known for several families of graphs, including trees (using arbitrary rings) \cite{GPT} and cycles (using $\mathbb{Z}$) \cite{HMR}.  Handschy-Melnick-Reinders gave an explicit formula for the smallest value of each element of their flow-up basis \cite{HMR}.   Bowden-Hagen-King-Reinders gave an explicit formula for each entry of each basis element of splines on cycles over $\mathbb{Z}$ under the hypothesis that one pair of adjacent edge-labels is relatively prime \cite{BHKR}.  (The general GKM package implies each pair of adjacent edge-labels is relatively prime.)
\item With a basis, we can try to explicitly describe multiplication tables as in Section \ref{subsection: Grassmannians}.  Even in the case of integers and integers mod $m$ this can be a challenge.  Bowden-Hagen-King-Reinders used their basis for splines on some cycles to write an explicit multiplication table \cite{BHKR}.  Bowden-Philbin-Swift-Tammaro  gave an explicit formula for the multiplication table of splines mod $m$ over arbitrary graphs, using a more implicitly defined basis \cite{BPST}.  
\end{itemize}

Much of this work extends to PIDs in general, including $\mathbb{C}[x]$ and $\mathbb{R}[x]$.  Some of it even extends to $\mathbb{Z}[x]$ in the geometric context even though $\mathbb{Z}[x]$ is a quintessential example of a ring that is not a PID.  Indeed if $X$ is a GKM space for which GKM theory can be used with {\em integer} coefficients (like the flag variety), then we can restrict the torus  $T$ to various one-dimensional subgroups and get a map from splines over $\mathbb{Z}[t_1,\ldots,t_n]$ to $\mathbb{Z}[x]$.  Judicious choice of the one-dimensional subgroup together with the GKM conditions imply that each edge-label is an ideal of the form $nx\mathbb{Z}$ for some integer $n$.  The ring of splines in that case is very closely related to the ring of splines over integers obtained by evaluating at $x=1$ as in \cite[Theorem 2.12]{GPT}.  

\subsection{Graph operations}

One of the most interesting aspects of generalized splines is that natural graph operations also give useful information.  In this section we will describe ways to remove parts of a graph.  This often changes the underlying graph in ways that violate conditions from classical splines or geometry.  In a simple case of our first example, removing a vertex together with its incident edges leaves a graph that no longer satisfies the regularity conditions that moment graphs of manifolds satisfy.  In the second example, removing an edge typically results in a graph that is not dual to a triangulation.

Let us first consider the effects on $R_{G,\alpha}$ of removing a subgraph of $G$.   Suppose $G'$ is a subgraph of $G$.  Let $M_{G/G'}$ be the collection of splines that are zero on all of the vertices of $G'$, namely
\[M_{G/G'} = \{p \in R_{G}: p(u)=0 \textup{ for all } u \in V(G)-V(G')\}\]  
Define the restriction map $\phi: R_{G} \stackrel{}{\rightarrow} R_{G'}$ that forgets all vertices in $G$ but not $G'$, namely
\[(p_u)_{u \in V(G)} \mapsto (p_u)_{u \in V(G')}\]
Then $M_{G/G'}$ is the kernel of $\phi$.  In other words
\[M_{G/G'} \hookrightarrow R_{G} \stackrel{\phi}{\rightarrow} R_{G'}\]
is exact.  Moreover the first isomorphism theorem tells us
\[\textup{Im} \phi \cong R_{G'}/M_{G'/G}\]
Colloquially $\textup{Im }\phi$ tells us which splines can be extend from $G'$ to $G$.  When $G'$ consists of a single vertex, this corresponds to adding a single triangle to a triangulation. 

We can also remove edges from the graph $G$, effectively removing the conditions between two vertices.  If $G'$ is a graph obtained from $G$ by erasing some edges then we get an inclusion   $R_{G} \hookrightarrow R_{G'}$.  For instance if $G$ is planar graph, we can choose a vertex as the root and take a spanning subtree $T \subseteq G$ for which the distance between two vertices in $T$ is the same as in $G$. The injection $R(G) \hookrightarrow R(T)$ describes $R(G)$ as a submodule of a free module whose generators are well-understood \cite[Theorem 4.1]{GPT}.  Refining the way edges are added to $T$ gives a sequence of inclusions 
\[R(G) \hookrightarrow R(G_i) \hookrightarrow \cdots \hookrightarrow R(G_1) \hookrightarrow R(T)\] 
This can constrain the codimension of each inclusion in interesting ways.   We know of no good analogue to this operation for splines over triangulations: removing an edge appears to ``cut" the plane so that two triangles are no longer adjacent, without moving or removing those triangles.  

We conclude with another example, an algorithm that uses graph operations like contractions and deletions to produce explicit bases for splines over $\mathbb{Z}$.  The algorithm has the following steps:
\begin{enumerate}
\item If $G$ is a tree then we have an explicit formula for a flow-up basis for its splines \cite[Theorem 4.1]{GPT}.
\item Otherwise suppose $v$ is a vertex in $G$ with neighbors $N(v)$.  Construct the (multi)graph $G_v'$ by erasing $v$ and any incident edges from $G$ and then add all possible edges $uu'$ for $u \neq u'$ with $u,u' \in N(v)$.  The edge $uu'$ is labeled with the ideal $\alpha_{uv} \oplus \alpha_{u'v}$. 
\item If $G_v'$ has any multiedges replace it with the graph $G_v$ constructed as follows.  If the pair of vertices $uu'$ have edges labeled $I_1$, $I_2$, \ldots, $I_k$ then replace those edges with a single edge labeled $I_1 \cap I_2 \cap \cdots I_k$. 
\item Return to Step 1 with the graph $G_v$.
\end{enumerate}
The algorithm eventually terminates since the graph with a single vertex is a tree.  The core of the proof of the algorithm is that the operations of $\oplus$ and $\cap$ distribute so that $R_G$ surjects onto $R_{G_v}$.  In fact this surjection holds for all PIDs (including $\mathbb{C}[x]$ and $\mathbb{R}[x]$) and the algorithm holds for all PIDs.  Figure \ref{figure: algorithm using graph operations} shows the graph reductions described by this algorithm on a complete graph using the integers.  In this case $\langle m \rangle \oplus \langle m' \rangle$ is just $\gcd(m,m')$ while $\langle m \rangle \cap \langle m' \rangle = \textup{lcm}(m,m')$.
\begin{figure}[h]
\scalebox{0.8}{
\begin{picture}(100,130)(0,0)
\put(50,0){\circle*{5}}
\put(100,50){\circle*{5}}
\put(0,75){\circle*{5}}
\put(50,125){\circle*{5}}

\put(50,0){\line(1,1){50}}
\put(50,0){\line(-2,3){50}}
\put(50,0){\line(0,1){125}}

\put(100,50){\line(-4,1){100}}
\put(100,50){\line(-2,3){50}}

\put(0,75){\line(1,1){50}}

\put(18,25){4}
\put(85,25){8}
\put(75,45){6}
\put(85,80){7}
\put(25,110){3}
\put(40,95){6}
\end{picture}
}
\hspace{0.25in} \begin{picture}(10,10)(0,0)\put(0,30){$\longrightarrow$} \end{picture} \hspace{0.5in}
\scalebox{0.8}{
\begin{picture}(100,130)(0,0)
\put(50,0){\circle*{8}}
\put(100,50){\circle*{8}}
\put(0,75){\circle*{8}}

\put(50,-2){\line(1,1){50}}
\put(50,2){\line(-2,3){50}}
\put(50,2){\line(1,1){50}}
\put(50,-2){\line(-2,3){50}}

\put(100,48){\line(-4,1){100}}
\put(100,52){\line(-4,1){100}}


\put(35,28){4}
\put(65,25){8}
\put(45,50){6}

\put(65,25){8}
\put(25,77){$\langle 3 \rangle \oplus \langle 7 \rangle = \langle 1 \rangle$}
\put(-21,28){$\langle 3 \rangle \oplus \langle 6 \rangle$}
\put(5,8){$=\langle 3 \rangle$}
\put(85,23){$\langle 7 \rangle \oplus \langle 6 \rangle$}
\put(93,8){$=\langle 1 \rangle$}

\end{picture}
}
\hspace{0.5in} \begin{picture}(10,10)(0,0)\put(0,30){$\longrightarrow$} \end{picture} \hspace{0.35in}
\scalebox{0.8}{
\begin{picture}(100,130)(0,0)
\put(50,0){\circle*{8}}
\put(100,50){\circle*{8}}
\put(0,75){\circle*{8}}

\put(50,0){\line(1,1){50}}
\put(50,0){\line(-2,3){50}}

\put(100,50){\line(-4,1){100}}


\put(35,28){12}
\put(65,25){8}
\put(45,50){6}

\end{picture}
}

\vspace{-.5in}\hspace{1.85in} \begin{picture}(10,10)(0,0)\put(0,20){$\longrightarrow$} \end{picture} \hspace{-0.25in}
\scalebox{0.8}{
\begin{picture}(100,130)(0,0)
\put(50,0){\circle*{8}}
\put(100,50){\circle*{8}}

\put(50,-2){\line(1,1){50}}
\put(50,2){\line(1,1){50}}

\put(65,25){8}
\put(80,20){$\langle 6 \rangle \oplus \langle 12 \rangle$}
\put(88,10){$= \langle 6 \rangle$}
\end{picture}
}
\hspace{0.5in} \begin{picture}(10,10)(0,0)\put(0,20){$\longrightarrow$} \end{picture} \hspace{-0.25in}
\scalebox{0.8}{
\begin{picture}(100,130)(0,0)
\put(50,0){\circle*{8}}
\put(100,50){\circle*{8}}

\put(50,0){\line(1,1){50}}

\put(60,25){24}
\end{picture}
}
\caption{Computing a basis for splines over $\mathbb{Z}$}\label{figure: algorithm using graph operations}
\end{figure}
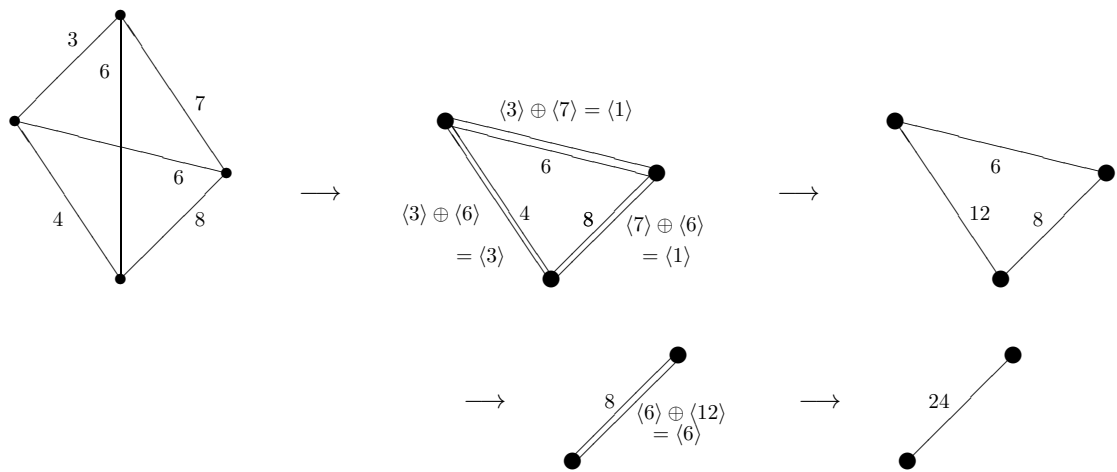
The final graph is an edge so its flow-up basis is $\{[1,1], [0,24]\}$.  The basis for the triangle surjects onto this basis, with flow-up basis itself of $\{[1,1,1], [0,24,24], [0,0,12]\}$.  (The last entry in the second basis vector requires a small calculation.)  The basis for the square surjects onto this basis, with flow-up basis itself of $\{[1,1,1,1], [0,24,24,24], [0,0,12,42],[0,0,0,42]\}$. 

\section{Acknowledgements}

The author is exceptionally grateful to Tanya Sorokina, Larry Schumaker, Hal Schenck, Alexei Kolesnikov and Gwyneth Whieldon for useful discussions, and the Mathematische Forschungs Institute Oberwolfach for sponsoring a productive workshop.  This work was partially supported by NSF grants DMS--1362855 and DMS--1248171.

\def\cprime{$'$}

\end{document}